\documentclass[11pt]{amsart}
\usepackage[utf8x]{inputenc}
\usepackage[english]{babel}
\usepackage[T1]{fontenc}
 
\usepackage{graphicx}
\usepackage{caption}
\usepackage{subcaption}

\usepackage{amssymb,amsmath}
\usepackage{mathtools}

\usepackage[hidelinks]{hyperref}
\usepackage{indentfirst}
\usepackage{enumerate,amsmath,amssymb, mathrsfs,mathtools}
\usepackage{appendix}
\usepackage{latexsym}
\usepackage{url}
\usepackage{color}
\usepackage{accents}
\usepackage{setspace}
\usepackage{pdfpages}
\usepackage{stmaryrd}

\usepackage{amsrefs}

\usepackage[margin=2.5cm]{geometry}
\allowdisplaybreaks

\newcommand{\ubar}[1]{\underaccent{\bar}{#1}}
\BibSpec{article}{%
	+{}{\PrintAuthors} {author}
	+{,}{ } {title}
	+{, }{\textit } {journal}
	+{}{ \parenthesize} {date}
		+{,  }{no. } {volume}
	+{,}{ } {pages}
	+{,}{ } {note}
}
\usepackage{bbm}

\ExplSyntaxOn

\NewExpandableDocumentCommand{\gobblefirst}{m}
{
	\tl_tail:n { #1 }
}

\ExplSyntaxOff
\BibSpec{book}{%
	+{}{\PrintAuthors}  {author}
	+{. }{}{title}
	+{,}{ }{series}
	+{,}{ vol.~}{volume}
	+{. }{\textit}{publisher}
	+{,}{ ISBN \gobblefirst}{isbn}
}
\BibSpec{collection.article}{%
	+{}{\PrintAuthors}{author}
	+{, }{}{title}
	+{, }{\textit}{booktitle}
	+{, }{ \DashPages}{pages}
	+{,}{ }{series}
	+{, }{}{volume}
	+{, }{\textit}{publisher}
	+{,}{ }{date}
}

\mathcode`l="8000
\begingroup
\makeatletter
\lccode`\~=`\l
\DeclareMathSymbol{\lsb@l}{\mathalpha}{letters}{`l}
\lowercase{\gdef~{\ifnum\the\mathgroup=\m@ne \ell \else \lsb@l \fi}}%
\endgroup

\def\XXint#1#2#3{{\setbox0=\hbox{$#1{#2#3}{\int}$ }
		\vcenter{\hbox{$#2#3$ }}\kern-.6\wd0}}

\newtheorem{prop}{Proposition}
\newtheorem{thm}[prop]{Theorem}
\newtheorem{lem}[prop]{Lemma}
\newtheorem{coro}[prop]{Corollary}

\newtheorem{rema}[prop]{Remark}

\title[The Penrose inequality in extrinsic geometry]{ 
	The Penrose inequality in extrinsic geometry}
\author{Michael Eichmair}
\address{
	\textnormal{Michael Eichmair \newline  \indent
		University of Vienna \newline \indent
		Faculty of Mathematics  \newline \indent
		Oskar-Morgenstern-Platz 1 \newline \indent
		1090 Vienna, 	Austria  \newline\indent 
		 \href{https://orcid.org/0000-0001-7993-9536}{https://orcid.org/0000-0001-7993-9536} \newline\indent	
		 \href{mailto:michael.eichmair@univie.ac.atm}{michael.eichmair@univie.ac.at}}
}

\author{Thomas Koerber}
\address{\textnormal{Thomas Koerber  \newline \indent
		University of Vienna \newline \indent
		Faculty of Mathematics  \newline \indent
		Oskar-Morgenstern-Platz 1 \newline \indent 1090 Vienna,	Austria \newline\indent 
		 \href{https://orcid.org/0000-0003-1676-0824}{https://orcid.org/0000-0003-1676-0824} \newline \indent
		  \href{mailto:thomas.koerber@univie.ac.atm}{thomas.koerber@univie.ac.at}}
}

\begin{document}

		\date{\today}
	\onehalfspacing
\begin{abstract}
	The Riemannian Penrose inequality is a fundamental result in mathematical relativity. It has been a long-standing conjecture of G.~Huisken that an analogous result should hold in the context of extrinsic geometry. In this paper, we resolve this conjecture and show that the exterior mass $m$ of an asymptotically flat support surface $S\subset\mathbb{R}^3$ with nonnegative mean curvature and outermost free boundary minimal surface $D$ is bounded in terms of
	$$
	m\geq \sqrt{\frac{|D|}{\pi}}.
	$$ If equality holds, then the unbounded component of $S\setminus \partial D$  is a half-catenoid. 	In particular, this extrinsic Penrose inequality leads to a new characterization of the catenoid among all complete  embedded minimal surfaces with finite total curvature. 
To prove this result, we study minimal capillary surfaces supported on $S$ that minimize the  free energy and discover a quantity associated with these surfaces that is nondecreasing as the contact angle increases.
	\end{abstract}
	\maketitle 
 \section{Introduction}
 In mathematical models of isolated gravitational systems, asymptotically flat Riemannian 3-manifolds with nonnegative scalar curvature arise naturally as maximal initial data for the Einstein field equations. 
 The Riemannian Penrose inequality asserts that a fundamental quantity  $m_{ADM}$ associated with such initial data $(M,g)$ called the ADM-mass  (after R.~Arnowitt, S.~Deser, and C.~Misner \cite{ArnowittDeserMisner})  is bounded below by a quantity in terms of the area of the outermost minimal surface $N\subset M$, which serves as a proxy for the event horizon of the spacetime. More precisely, there holds
 $$
 m_{ADM}\geq \sqrt{\frac{|N|}{16\,\pi}}
 $$ 
with equality if and only if the unbounded component of $(M\setminus N,g)$ is isometric to the exterior of the so-called spatial Schwarzschild manifold. 
 The Riemannian Penrose inequality has been established by G.~Huisken and T.~Ilmanen \cite{HuiskenIlmanen} in the case where $N$ is connected using a weak notion of inverse mean curvature flow and, independently, by H.~Bray \cite{Bray2} in the  case where $N$ has one or multiple components using his conformal flow.  The theory of weak inverse mean curvature flow pioneered by G.~Huisken and T.~Ilmanen  has found many other  applications also outside the field of mathematical relativity; see, e.g., \cites{BrayNeves,GuanLi,BrendleHungWang,Xu,HuiskenKoerber}.
 
 In his dissertation under the guidance of G.~Huisken \cite{Volkmann}, A.~Volkmann has introduced the concept of an asymptotically flat support surface $S\subset \mathbb{R}^3$  and discovered several striking parallels between such  surfaces with nonnegative mean curvature and asymptotically flat Riemannian 3-manifolds with nonnegative scalar curvature. This includes the discovery of a scalar quantity associated with $S$ that stands in for the ADM-mass   for which a result analogous to the fundamental positive mass theorem \cite{SchoenYau,Witten} holds; see \cite[Theorem 2.8]{Volkmann}. A.~Volkmann has also proven a rigidity result for the existence of an unbounded stable free boundary minimal surface supported on $S$ that is similar in spirit to a rigidity result due to A.~Carlotto \cite{Carlotto} for the existence of unbounded stable minimal surfaces in asymptotically flat Riemannian 3-manifolds; see \cite[Theorem 2.13]{Volkmann}. 
 
 Also in view of these results, G.~Huisken \cite[p.~38]{Volkmann} has conjectured that a result analogous to  the Riemannian Penrose inequality should hold for asymptotically flat support surfaces. The goal of this paper is to confirm this conjecture. 
 
 To describe our contributions here, 	let $ S\subset \mathbb{R}^3$ be a connected  unbounded properly embedded surface such that $\mathbb{R}^3\setminus S$ has exactly two components. In particular, $S$ has no boundary.  We say that $S$ is an asymptotically flat support surface if its mean curvature is integrable and the complement of a compact subset of $S$ is contained in the graph of a function $\psi \in C^\infty(\mathbb{R}^2)$ such that, for some $\tau>1/2,$ 
 \begin{equation*} 
 	\begin{aligned} 
 		\sum_{i=1}^2|\partial_i\psi(y)|+\sum_{i,\,j=1}^2|y|\,|\partial^2_{ij}\psi(y)|=O(|y|^{-\tau}).
 	\end{aligned}
 \end{equation*} 
 Let $\nu(S)$ be the unit normal field of $S$ asymptotic to $-e_3$. The component of $\mathbb{R}^3\setminus S$ that $\nu(S)$ points out of is denoted by $M(S)$ and referred to as the region above $S$.  Let $H(S)$ be the mean curvature of $S$ computed as the tangential divergence of $\nu(S)$.

The exterior mass $m$ of $S$ is the quantity
 \begin{align} \label{exterior mass section 1}
 	m=\lim_{r\to\infty}\frac{1}{2\,\pi\,r}\,\int_{\left\{y\in\mathbb{R}^2:|y|=r\right\}}\sum_{i=1}^2\,y_i\,\partial_i\psi(y);
 \end{align}
 see \cite[p.~36]{Volkmann}. The integrability of $H(S)$ ensures that the limit on the right-hand side of \eqref{exterior mass section 1} exists; see \cite[Proposition 2.6]{Volkmann}.
 
  The following result of S.~Almaraz, E.~Barbosa, and L.~de Lima stands in for  the positive mass theorem for asymptotically flat Riemannian 3-manifolds; see \cite[Theorem 1.1]{AlmarazBarbosaDeLima} and \cite[Lemma 2.1]{Koerber}.
 \begin{thm} \label{extrinsic:pmt} Let $S\subset\mathbb{R}^3$ be an asymptotically flat support surface with nonnegative mean curvature. There holds $m\geq 0$  with equality if and only if $S$ is a flat plane. 
 \end{thm}
 \begin{rema}
 	A.~Volkmann has proven Theorem \ref{extrinsic:pmt} in the special case where $S$ is asymptotically catenoidal; see \cite[Definition 2.2]{Volkmann} and \cite[Theorem 2.8]{Volkmann}. 
 \end{rema}

		A nonempty compact minimal surface  $D\subset \bar M(S)$  that intersects $S$ orthogonally along $\partial D$ such that $\partial D\subset S$ bounds a bounded open subset of $S$ is called a free boundary minimal surface. Given a free boundary minimal surface $D\subset \bar M(S)$, let $S'\subset S$ be the closure of the unbounded component of $S\setminus \partial D$ and $M'(S)$ be the unbounded component of   $M(S)\setminus D$. We call $D$ outermost if every compact minimal surface  $\Sigma\subset \bar M'(S)$  that intersects $S$ orthogonally along $\partial \Sigma$ is contained in $D$. If an outermost free boundary minimal surface $D\subset \bar M(S)$ exists, we say that $S$ is an asymptotically flat support surface with  outermost minimal free boundary surface $D$, exterior surface $S'$, and exterior region $M'(S)$.
 
 As an example, we consider the half-catenoid $S'_m\subset \mathbb{R}^3$ of mass $m>0$ given by 
 \begin{align} \label{half-catenoid} 
S'_m= 	\{(y,x_3):y\in\mathbb{R}^2,\text{ }x_3\geq 0,\text{ and }|y|=m\,\cosh(m^{-1}\,x_3)\}.
 \end{align} 
Let $S_m\subset \mathbb{R}^3$ be an asymptotically flat support surface  such that $S'_m\subset S_m$ and $S_m\setminus S'_m\subset \{x\in \mathbb{R}^3:x_3<0\}$. Then $S_m$ is an asymptotically flat support surface with exterior mass $m$,  outermost  free boundary minimal surface $D_m=\{(y,0):y\in\mathbb{R}^2 \text{ and } |y|\leq m\}$, exterior surface $S'_m$, and exterior region $M'(S_m)=\{(y,x_3):y\in\mathbb{R}^2,\text{ } x_3>0, \text{ and } |y|<m\,\cosh(m^{-1}\,x_3)\}$; see Figure \ref{exterior_surface}. 

  	\begin{figure}\centering
	\includegraphics[width=\linewidth]{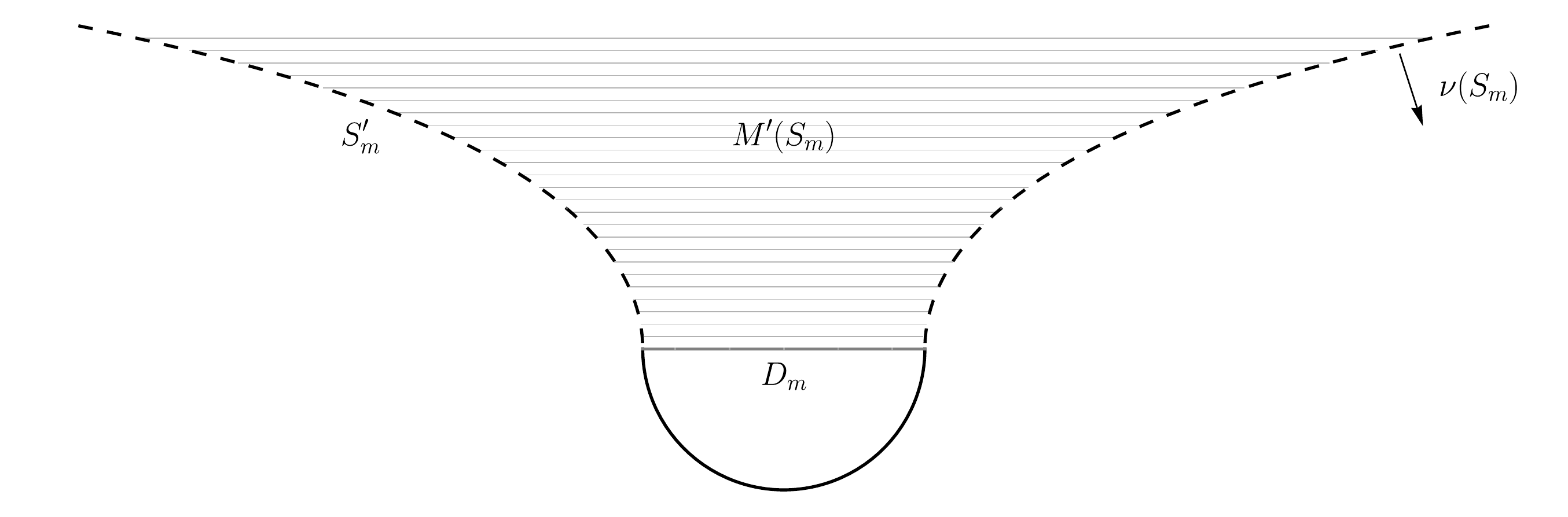}
	\caption{An illustration of the asymptotically flat support surface $S_m$. The outermost free boundary minimal surface $D_m$ is indicated by the solid gray line. The exterior surface $S_m'$ is indicated by the dashed black line. The exterior region $M'(S_m)$ is indicated by the hatched region.
	}
	\label{exterior_surface}
\end{figure}

In a similar way, every connected complete minimal surface embedded in $\mathbb{R}^3$ with finite total curvature that is not a plane gives rise to an asymptotically flat support surface with outermost free boundary minimal surface; see \cite[Proposition 1]{Schoen}. After D.~Costa's discovery of one such surface that is distinct from both the plane and the catenoid, such surfaces have been shown to exist in abundance; see, e.g., \cite{Traizet,Kapouleas}. Note that, in many cases, the outermost free boundary minimal surface of these examples has multiple components; see, e.g., \cite{HoffmanKarcher}.
 
  The main result of this paper is the following positive resolution of G.~Huisken's conjecture as stated in \cite[p.~38]{Volkmann}. 
 \begin{thm} \label{extrinsic:rpi} Let $S\subset \mathbb{R}^3$ be an asymptotically flat support surface  with exterior mass $m$, outermost free boundary minimal surface $D\subset \bar M(S)$, and exterior surface $S'\subset S$. Suppose that $H(S)\geq 0$ on $S'$. There holds 
 	\begin{align*} 
 	m\geq\sqrt{\frac{|D|}{\pi}}
 	\end{align*} 
 	with equality if and only if $S'$ is a translation of the half-catenoid \eqref{half-catenoid} of mass $m$.
 \end{thm}
 \begin{rema}
 	Note that $D$ is not required to be connected in Theorem \ref{extrinsic:rpi}.
 \end{rema}
 \begin{rema}
 	The assumption that $D$ is outermost in Theorem \ref{extrinsic:rpi} cannot be dropped. Indeed, it is straightforward to construct the extension $S_1$ of the half-catenoid \eqref{half-catenoid} with mass $m=1$ so that $S_1$ is axially symmetric, has nonnegative mean curvature, and such that there is a free boundary minimal surface contained in $\bar M(S)\setminus \bar M'(S)$ with area greater than $\pi$.
 \end{rema}
 \begin{rema}
 	The assumption that $H(S)\geq 0$ on $S'$ in Theorem \ref{extrinsic:rpi} cannot be dropped. Indeed, an asymptotically flat support surface  obtained as a perturbation of the extension $S_1$ of the half-catenoid \eqref{half-catenoid} with mass $m=1$  in direction of $-e_3$ that is supported outside a sufficiently large compact set has the same outermost free boundary minimal surface as $S_1$ but exterior mass less than $m=1$. Note that the mean curvature of such a surface cannot possibly be nonnegative. 
 \end{rema}

 \begin{rema}
 	A.~Volkmann \cite[Proposition 2.9]{Volkmann} has confirmed Theorem \ref{extrinsic:rpi} in the special case where $S'$ is graphical, asymptotically catenoidal,  and the inward co-normal of $\partial S'\subset S'$ is $e_3$.
 \end{rema}
 \begin{rema}
 	In \cite{SchoenYau3}, using a conformal method, R.~Schoen and S.-T.~Yau have shown that the proof of the positive mass theorem can be reduced to the case where the asymptotically flat Riemannian 3-manifold is asymptotic to Schwarzschild initial data. By contrast, we are not aware of a method that reduces the proof of Theorem \ref{extrinsic:rpi} to the case where $S$ is asymptotically catenoidal.
 \end{rema}
 \begin{rema}
 	We will explain in Section \ref{section:proof of main results 2} how the method employed to prove Theorem \ref{extrinsic:rpi} can be used to give an alternative proof of Theorem \ref{extrinsic:pmt}.
 \end{rema}
In the global theory of minimal surfaces, the catenoid 
$$
\{(y,x_3):y\in\mathbb{R}^2,\text{ }x_3\in\mathbb{R},\text{ and }|y|=\cosh(x_3)\}
$$
admits many beautiful characterizations. Up to scaling and rigid motions,   the catenoid is the only complete nonplanar minimal surface of revolution \cite{Bonnet} and the unique complete embedded minimal surface with finite total curvature and two ends \cite{Schoen,LopezRos}. The catenoid has been characterized in terms of its total curvature \cite{Osserman} and in terms of its Morse index \cite{LopezRos1}. We will now explain how 
Theorem \ref{extrinsic:rpi} provides a new characterization of the catenoid in terms of its flux and  neck size. Recall from, e.g., \cite[Proposition 1]{Schoen} that every connected complete  minimal surface embedded in $\mathbb{R}^3$ with finite total curvature  has a finite number of ends. These ends can be ordered in terms of their height. Moreover, recall from, e.g., \cite[\S17]{Fang} that, as a consequence of the first variation formula, each of these ends admits a vector called the flux that is a conserved quantity of that end; see \eqref{Flux}. We refer to the largest value of the length of these flux vectors as the largest flux of the minimal surface. If the minimal surface is not a plane, we refer to $\sqrt{4\,\pi}$ times  the square root  of the maximum of the least amount of area needed to separate the region above the top end into two unbounded components and the least amount of area needed to separate the region below the bottom end into two unbounded components as the neck size of the minimal surface; see \eqref{neck size}. Moreover, we agree that the neck size of a plane is zero. Note that the largest flux of the catenoid equals its neck size; see also \cite[p.~481]{KorevaarKusnerSolomon}. As a consequence of Theorem \ref{extrinsic:rpi}, we obtain that the plane and the catenoid are the only connected complete minimal surfaces embedded in $\mathbb{R}^3$ with finite total curvature that satisfy this property. 
\begin{thm} \label{catenoid characterization}
	Up to scaling and rigid motions, the plane and the catenoid are the only connected complete minimal surfaces embedded in $\mathbb{R}^3$ with finite total curvature whose largest flux is no larger than their neck size. 
\end{thm}

 \subsection*{Outline of related results} To prove the Riemannian Penrose inequality, G.~Huisken and T.~Ilmanen \cite{HuiskenIlmanen} have developed the theory of weak solutions for inverse mean curvature flow emerging from an outermost minimal surface. Building on an observation due to R.~Geroch \cite{Geroch}, they have shown that the so-called Hawking mass is nondecreasing along this flow provided that the scalar curvature of the asymptotically flat Riemannian 3-manifold is nonnegative and the evolving surface is connected; see \cite[\S5]{HuiskenIlmanen}.
 
   In an attempt to adapt this strategy to prove Theorem \ref{extrinsic:rpi},  T.~Marquardt \cite{Marquardt} has studied the inverse mean curvature flow of an evolving  surface that is required to intersect an asymptotically flat support surface $S\subset \mathbb{R}^3$ orthogonally.   Along this flow, T.~Marquardt has observed that a modified Hawking mass  is nondecreasing  provided that $S$ has nonnegative mean curvature  and that the evolving surface is connected. The second-named author \cite{Koerber} has taken this approach further and obtained the following partial result towards Theorem \ref{extrinsic:rpi}.
\begin{thm}[{\cite[Corollary 1.4]{Koerber}}]\label{extrinsic:imcf}Let $S\subset \mathbb{R}^3$ be an asymptotically flat support surface  with exterior mass $m$, connected outermost free boundary minimal surface $D\subset \bar M(S)$, and exterior surface $S'\subset S$. Suppose that $H(S)\geq 0$ on $S'$. There holds 
	\begin{align} \label{imcf estimate} 
	m\geq\sqrt{\frac{|D|}{2\,\pi}}.
	\end{align} 
\end{thm} 
In view of Theorem \ref{extrinsic:rpi}, estimate \eqref{imcf estimate} is not optimal. This is because the modified Hawking mass considered in \cite{Marquardt} is strictly increasing along the inverse mean curvature flow considered in \cite{Marquardt,Koerber} even in the case where $S'$ is a half-catenoid \eqref{half-catenoid}. This suggests that inverse mean curvature flow is not an adequate technique  
to prove Theorem \ref{extrinsic:rpi}; see also \cite[p.~321]{Koerber}.

As an alternative approach towards Theorem \ref{extrinsic:rpi}, the authors \cite{EichmairKoerber2} have developed a gluing technique that can be applied to an asymptotically flat support surface. Combining \cite[Theorem 8]{EichmairKoerber2}, \cite[Lemma 2.1]{Koerber}, and the Riemannian Penrose inequality for asymptotically flat Riemannian 3-manifolds \cite{Bray2} shows that the assumption that $D$ be connected in Theorem \ref{extrinsic:imcf} can be dispensed with. Again, this approach is not suited to improve upon estimate  \eqref{imcf estimate}.

\subsection*{Outline of our approach} The preceding discussion suggests that an altogether new approach is needed to prove Theorem \ref{extrinsic:rpi}. Our strategy here is motivated by the fact that the exterior region $M'(S_1)$ of an asymptotically flat support surface $S_1$ that contains the half-catenoid  \eqref{half-catenoid} with mass $m=1$ is foliated by the family $\{\Sigma_t:t\geq 0\}$ of flat disks $\Sigma_t=\{x\in \bar M(S_1):x_3=t\}$ and that, for each $t\geq 0$, $\Sigma_t$ intersects $S_1$ at a constant angle of  $\arccos(-\tanh (t))$ and satisfies 
\begin{align} \label{catenoid constant} 
\frac{1}{\cosh (t)}\,\sqrt{\frac{|\Sigma_t|}{\pi}}=1.
\end{align} 

 Let $S$ be an asymptotically flat support surface with exterior mass $m$,  outermost free boundary minimal surface $D\subset \bar M(S)$, and exterior region $S'\subset S$. By Lemma \ref{topology}, we assume that $S$ is diffeomorphic to the plane.  We say that a compact surface $\Sigma\subset \bar M(S)$ is admissible if $\partial \Sigma\subset S$, the intersection of $\Sigma$ and  $S$ is transverse, and if $\Sigma$ has no closed components contained in $M(S)$.   Given an admissible surface $\Sigma\subset \bar M(S)$, let $S(\Sigma)$ denote the closure of the union of the bounded components of $S\setminus \partial \Sigma$ and $\Omega(\Sigma)$ be the  union of the bounded components of $M(S)\setminus \Sigma$; see Figure \ref{lateral surface}. 
   	\begin{figure}\centering
 	\includegraphics[width=\linewidth]{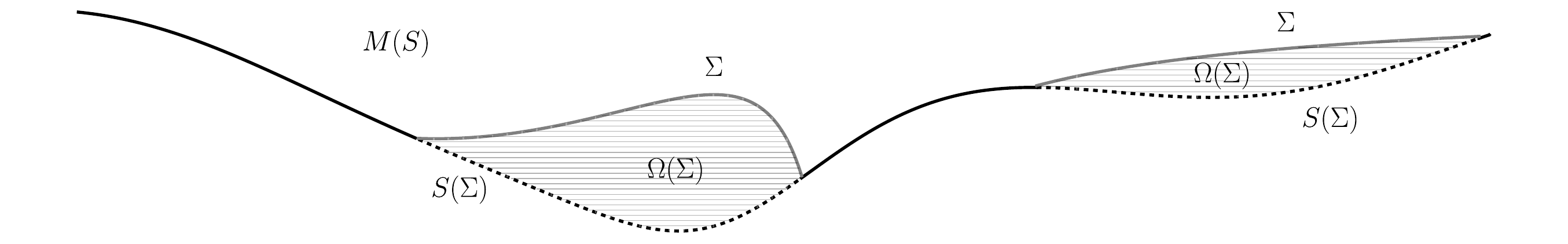}
 	\caption{An illustration of an admissible surface $\Sigma$ indicated by the solid gray line with two components supported on an asymptotically flat support surface $S$ indicated by the solid black line. The lateral surface $S(\Sigma)$ of $\Sigma$ is indicated by the dashed black line. The inside $\Omega(\Sigma)$ of $\Sigma$ is indicated by the hatched region. 
 	}
 	\label{lateral surface}
 \end{figure}
 
 We consider the free energy profile $\rho:[|S(D)|,\infty)\to[|D|,\infty)$ given by 
 \begin{align} \label{free energy profile intro} 
 	\rho(z)=\inf\{|\Sigma|:\Sigma\subset \bar M(S)\text{ is admissible such that $\Omega(D)\subset \Omega(\Sigma)$ and  $|S(\Sigma)|=z$}\}.
 \end{align} 
Assume first that $\rho$ is smooth. By the formulas for the first variation of area \eqref{area first} and first variation of enclosed area \eqref{enclosed area first},  if $z>|S(D)|$ and $\Sigma_z\subset \bar M(S)$ is an admissible surface  with $\Omega(D)\subset \Omega(\Sigma_z)$,  $|S(\Sigma_z)|=z,$ and $\rho(z)=|\Sigma_z|$, then $\Sigma_z$ is a  minimal capillary surface with capillary angle $\arccos(-\rho'(z))$, i.e.,  the mean curvature of $\Sigma_z$ vanishes and $\Sigma_z$ intersects $S$ at a constant angle of $\arccos(-\rho'(z))$. Using the second variation of area formula \eqref{area second} applied to  a constant normal speed variation and the Gauss-Bonnet theorem, we obtain 
	\begin{align} \label{rho estimate}
	\rho''(z)\leq |\partial \Sigma_z|^{-2}\,(1-\rho'(z)^2)\,\left(2\,\pi\,\chi(\Sigma_z)-\frac12\,\int_{\Sigma_z} |h(\Sigma_z)|^2-\sqrt{1-\rho'(z)^2}\,\int_{\partial \Sigma_z} H(S)\right);
\end{align}
see Lemma \ref{normal variation}. Here, $h(\Sigma_z)$ is the second fundamental form of $\Sigma_z$ and $\chi(\Sigma_z)$ its Euler characteristic. By the optimal isoperimetric inequality for minimal surfaces, we have $|\partial \Sigma_z|^2\geq 4\,\pi\,\rho(z)$; see \cite{Carleman,Almgren,Brendle}.  Using this and the assumption $H(S)\geq 0$,  we see that, in the case where $\Sigma_z$ is connected, \begin{align} \label{mono} \left[\rho(z)\,(1-\rho'(z)^2)\right]'\geq 0.\end{align} 
 Importantly, \eqref{mono} holds even in the case where $\Sigma_z$ has multiple components. To see this, we may instead apply the second variation of area formula \eqref{area second} to a suitable variation whose normal speed is only constant on the individual components of $\Sigma_{z}$; see Lemma \ref{comparison function}.  This observation hinges on the two important facts that both the area $|\Sigma|$ and the enclosed area $|S(\Sigma)|$ of an admissible surface $\Sigma\subset \bar M(S)$ scale quadratically and that the left-hand side of \eqref{mono} is a linear function of $\rho(z)$.   By Lemma \ref{minimization in homology class}, $\rho(|S(D)|)\,(1-\rho'(|S(D)|)^2)=|D|$.   In view of \eqref{catenoid constant}, we expect that 
$$\lim_{z\to\infty}\rho(z)\,(1-\rho'(z)^2)= \pi\,m^2.$$   This suggests that the free energy profile provides a link between $\pi\,m^2$ and $|D|$.

 We note that H.~Bray \cite{bray} has suggested a related approach based on an analysis of the isoperimetric profile to prove the Riemannian Penrose inequality and that C.~Li \cite{Li} has studied a variational problem  related to \eqref{free energy profile intro} to prove  a polyhedron comparison result conjectured by M.~Gromov.

To make the preceding approach rigorous, we first  show that, for every $t>0$, there exists a minimal capillary surface $\Sigma_t\subset \bar M(S)$ with $\Omega(D)\subset\Omega(\Sigma_t)$ that minimizes
the free energy 
\begin{align} \label{free energy intro} 
	J_t(\Sigma)=|\Sigma|-\tanh (t)\,|S(\Sigma)|
\end{align} 
among all admissible surfaces $\Sigma\subset \bar M(S)$ with $\Omega(D)\subset \Omega(\Sigma)$; see Proposition \ref{capillary existence}. Note that, necessarily, $\rho(|S(\Sigma_t)|)=|\Sigma_t|$ for every $t>0$. To this end, we construct barriers for \eqref{free energy intro} in the asymptotic region of $\bar M(S)$, see Proposition \ref{prop:sub}, as well as near $D$, see Proposition \ref{prop:super}, and apply a result of J.~Taylor \cite{Taylor}. The construction in Proposition \ref{prop:sub}  relies on several global results from minimal surface theory such as the half-space theorem of D.~Hoffman and W.~Meeks \cite{HoffmanMeeks}. While, given $t>0$, the surface $\Sigma_t$ may not be unique as a minimizer of \eqref{free energy intro}, we show that the enclosed area $|S(\Sigma_t)|$ is always a strictly increasing function of $t$ and that $\Sigma_t$ diverges as $t\to\infty$; see Lemma \ref{s property 1} and Lemma \ref{inner divergence}.

In view of \eqref{rho estimate}, we introduce the quantity 
\begin{align} \label{intro free energy mass} 
m_f(\Sigma_t)=\frac{1}{\cosh (t)}\,\sqrt{\frac{|\Sigma_t|}{\pi}}
\end{align}
associated with $\Sigma_t$, which we dub the free energy mass of $\Sigma_t$. 
 Building on the observation \eqref{mono} and using an argument based on  comparison functions as in \cite{bray}, we show that \eqref{intro free energy mass} is nondecreasing for $t>0$, even in the case where $\Sigma_t$ has multiple components; see Proposition \ref{montonicity prop}.    Since $D$ is outermost, we have $\Sigma_t\to D$ smoothly as $t\to0$ so that the monotonicity extends to the case where $t\geq 0$; see Corollary \ref{monotonicity corollary}.
  
We then study the asymptotic behavior of the free energy mass \eqref{intro free energy mass}. Using the optimal isoperimetric inequality for minimal surfaces  again and  integration by parts as well as  the coarse area estimate Lemma \ref{coarse area estimate}  and that $\Sigma_t$ diverges as $t\to\infty$ to control the error terms arising in the computation,  we show that $m_f(\Sigma_t)\to m$ as $t\to\infty$; see Proposition \ref{asymptotics prop}.  In conjunction with Corollary \ref{monotonicity corollary}, using that $$m_f(D)=\sqrt{\frac{|D|}{\pi}},$$ Theorem \ref{extrinsic:rpi} follows. 

Finally, we consider the case of equality. By the arguments leading to \eqref{mono}, we see that, for all $t>0$, $\Sigma_t$ is a union of flat disks and that $H(S)=0$ on $\partial \Sigma_t$. We observe that, in this case, the surfaces $\Sigma_t$, $t>0$,  are mutually disjoint. It then follows that every $\Sigma_t$ is connected and that $S'$ is contained in a minimal surface immersed in $\mathbb{R}^3$ with total curvature equal to $8\,\pi$. By a result of R.~Osserman \cite{Osserman}, such a minimal surface is necessarily a catenoid.

We remark that a related  strategy only works in H.~Bray's approach \cite{bray} to prove the Riemannian Penrose inequality when requiring the substantial additional assumption \cite[Condition 2]{bray}. This is because the Hawking mass is not additive and the volume enclosed by an isoperimetric surface scales cubically rather than quadratically.

\subsection*{Acknowledgments}
The authors would like to sincerely thank Simon Brendle both for pointing out to them the potential relevance of minimal capillary surfaces in this problem and for his valuable feedback on a draft of this paper that has led to significant simplifications of some of the proofs. Part of this research was carried out while the authors were visiting Columbia University. They gratefully acknowledge the  hospitality during their stays. 
%acknowledgments eichmair-chodosh-wang-brendle-johne-hirsch-huisken
This research was funded in part  by the Austrian Science Fund (FWF) [10.55776/M3184, 10.55776/Y963]. %For open access purposes, the authors have applied a CC BY public copyright license to an author accepted manuscript version arising from this submission.

\section{Sub- and supersolutions}

In this section, we assume that $ S\subset \mathbb{R}^3$ is a properly embedded plane with integrable mean curvature. 
Moreover, we assume that the complement of a compact subset of $S$ is contained in the graph of a function $\psi \in C^\infty(\mathbb{R}^2)$ such that, for some $\tau>1/2$,
\begin{equation} 
	\begin{aligned} \label{asymptotically flat section 3}
		\sum_{i=1}^2|\partial_i\psi(y)|+\sum_{i,\,j=1}^2|y|\,|\partial^2_{ij}\psi(y)|=O(|y|^{-\tau}).
	\end{aligned}
\end{equation} 
Let $\nu(S)$ be the unit normal field of $S$ asymptotic to $-e_3$.  The component of $\mathbb{R}^3\setminus S$ that $\nu(S)$ points out of is denoted by $M(S)$ and referred to as the region above $S$.   We denote the second fundamental form of $S$ with respect to $\nu(S)$ by $h(S)$. Our convention is such that the trace of $h(S)$, i.e., the mean curvature $H(S)$, is the tangential divergence of $\nu(S)$.

The goal of this section is to construct sub- and supersolutions for the free energy \eqref{free energy intro}.

Recall from Appendix \ref{appendix:af surfaces} the definition of an admissible surface $\Sigma \subset \bar M(S)$ as well as that of its lateral surface $S(\Sigma)\subset S$ and  its inside $\Omega(\Sigma)\subset M(S)$. Moreover, recall the definitions of the normal $\nu(\Sigma)$ and the co-normal $\mu( \Sigma)$ as well as  the conventions for the geodesic curvature $k(\partial \Sigma)$, the second fundamental form $h(\Sigma),$ and  the mean curvature $H(\Sigma)$. Given $f\in C^\infty(\Sigma)$, recall the definitions of $\nabla^\Sigma f$ and $\Delta^\Sigma f$.

Recall from Appendix \ref{appendix:free energy} the definition of a minimal capillary surface $\Sigma\subset \bar M(S)$ with capillary angle $\theta\in (0,\pi)$ and the notion of stability for such a surface

We assume  that $S$ is not a flat plane. By \eqref{asymptotically flat section 3}, for every $\lambda>1$ sufficiently large, $S$ and $\{x\in\mathbb{R}^3:|x|=\lambda\}$ intersect transversely in a  Jordan curve $\Gamma_\lambda$ that is smoothly close to a horizontal circle.  By a result of A.~Korn \cite[\S1]{KornArthur} based on the implicit function theorem, there is a minimal graph  $\Sigma_\lambda\subset \mathbb{R}^3$ over a  strictly convex subset of $\mathbb{R}^2$   with $\partial \Sigma_\lambda=\Gamma_\lambda$. It follows that $\Sigma_\lambda$ has least area among all compact surfaces with boundary $\Gamma_\lambda$ and that $\Sigma_\lambda$ is the unique surface of least area with boundary $\Gamma_\lambda$; see also \cite{Rado}. Moreover, it follows from standard elliptic theory that  
\begin{align} \label{smoothly close} \lambda^{-1}\,\Sigma_\lambda \text{ is smoothly close to }\{(y,0):y\in\mathbb{R}^2\text{ and }|y|\leq1\};
\end{align}
see, e.g., \cite[Theorem 15.9]{GilbargTrudinger}. % By a theorem of T.~Rado \cite{Rado}, it follows that $\Sigma_\lambda$ is the only surface of least area with boundary $\Gamma_\lambda$; see also \cite[\S4.9]{Dierkes}.
  If $H(S)\geq 0$,   the maximum principle and \eqref{asymptotically flat section 3} imply that $\Sigma_\lambda \subset  \bar M(S)$ provided that $\lambda>1$ is sufficiently large.   
\begin{lem} \label{subsolutions} 
Suppose that $H(S)\geq 0$.	There is $\lambda_0>1$ with the following properties.
	\begin{itemize}
		\item[$\circ$]  $\Sigma_\lambda$ and $S$ intersect transversely along $\Gamma_\lambda$ for every $\lambda\geq \lambda_0$.
		\item[$\circ$] There holds $\Omega(\Sigma_{\lambda_1})\subset \Omega(\Sigma_{\lambda_2})$ for all $\lambda_2>\lambda_1\geq \lambda_0$.
		\item[$\circ$] $\{\Sigma_\lambda:\lambda\geq \lambda_0\}$ is a smooth foliation of $(\bar M(S)\setminus \bar \Omega(\Sigma_{\lambda_0}))\cup \Sigma_{\lambda_0}$. 
	\end{itemize}

\end{lem} 
\begin{proof}
If, for $\lambda>1$ sufficiently large,	$\Sigma_{\lambda}$ and $S$  intersect tangentially along $\Gamma_\lambda$ or  $\Sigma_\lambda$ touches $S$ at an interior point, then, by the strong maximum principle, $\Sigma_{\lambda}=S(\Sigma_{\lambda})$. If there was a sequence $\{\lambda_k\}_{k=1}^\infty$ with $\lambda_k\to\infty$ such that	either $\Sigma_{\lambda_k}$ and $S$  intersect tangentially along $\Gamma_{\lambda_k}$ or  $\Sigma_{\lambda_k}$ touches $S$ at an interior point, then  $S$ is  area-minimizing.  Since $S$ is not a flat plane by assumption, this is a contradiction.
	
	If $\lambda>1$ is sufficiently large, using that $\Sigma_\lambda$ is area-minimizing and the maximum principle, we see that $\Sigma_\lambda\setminus \Gamma_\lambda\subset \Omega(\Sigma_{\lambda'})$ for every $\lambda'>\lambda$. 
	Let $x\in M(S)$ and suppose, for a contradiction, that $x\notin \Omega(\Sigma_\lambda)$ for every $\lambda>1$. Using standard tools from minimal surface theory, we see that there is a sequence $\{\lambda_k\}_{k=1}^\infty$ with $\lambda_k\to\infty$ such that $\Sigma_{\lambda_k}$ converges locally smoothly to an unbounded area-minimizing boundary. Again, such a boundary must be a flat plane. Since $H(S)\geq 0$, the proof of the  half-space theorem \cite[Theorem 1]{HoffmanMeeks} applies to show that $S$ is also a flat plane, a contradiction.
	
	Next, note that for $\lambda>1$ sufficiently large, $\Sigma_{\lambda'}$ converges to an area-minimizing  surface $\Sigma'_\lambda$ with $\partial \Sigma'_\lambda=\Gamma_\lambda$ as $\lambda'\to \lambda$. By uniqueness, we have $\Sigma_\lambda=\Sigma'_{\lambda}$.  Moreover, it follows from \eqref{smoothly close} that there is $\kappa>0$ such that, for all $\lambda>1$ sufficiently large, 
	$$
	\int_{\Sigma_\lambda} |\nabla^{\Sigma_\lambda} f|^2-\int_{\Sigma_\lambda} |h(\Sigma_\lambda)|^2\,f^2\geq \frac{\kappa}{\lambda^2}\,\int_{\Sigma_\lambda} f^2
	$$
	for every $f\in C^\infty(\Sigma_\lambda)$ with $f=0$ on $\partial \Sigma_\lambda$. In fact, $\kappa=5$ will do. Arguing as in the proof of Lemma \ref{local foliation}, using that the intersection of $\Sigma_\lambda$ and $S$ is transverse,  we see that  $\{\Sigma_\lambda:\lambda\geq \lambda_0\}$ is a smooth foliation. The assertion follows.
\end{proof}
\begin{prop} \label{prop:sub}
Suppose that $H(S)\geq 0$. Let $t_0\in \mathbb{R}$.	There exists $v\in C^\infty(\bar M(S))$  such that
	\begin{equation} \label{subsolution properties} 
	\begin{aligned}
		&\circ \qquad v(x)\to\infty \text{ as }x\to\infty,\\
		&\circ \qquad Dv(x)\neq 0 \text{ in } \{x\in \bar M(S) : v(x)\geq  t_0\},\\
		&\circ \qquad \operatorname{div}(|Dv|^{-1}\,Dv)=0\text{ in }\{x\in M(S): v(x)\geq  t_0\},\text{ and} \\	
		&\circ \qquad |D v|^{-1}\,\langle Dv,\nu(S)\rangle <-\tanh (v)\text{ on }\{x\in S:v(x)\geq t_0\}. 	
	\end{aligned}
	\end{equation} 
\end{prop}
\begin{proof}
	By  Lemma \ref{subsolutions}, the function $\tilde v: \bar M(S)\setminus \bar \Omega(\Sigma_{\lambda_0+1})\to\mathbb{R}$ determined by $\tilde v(x)=\lambda$ if $x\in \Sigma_\lambda$ is smooth.  We extend $\tilde v$ to a smooth function  on all of $\bar M(S)$ so that $\tilde v(x)< \lambda_0+1$ in $\bar\Omega(\Sigma_{\lambda_0+1})\setminus \Sigma_{\lambda_0+1}$. Since $\{\Sigma(\lambda):\lambda\geq \lambda_0+1\}$ is a smooth foliation, we have $\tilde v(x)\to\infty$ as $x\to\infty$ and  $D\tilde v(x)\neq 0$ in  $\{x\in \bar M(S):\tilde v(x)\geq \lambda_0+1\}$. Since $H(\Sigma_\lambda)=0$ for every $\lambda\geq \lambda_0+1$, there holds $\operatorname{div}(|D\tilde v|^{-1}\,D\tilde v)=0$ in  $\{x\in M(S): \tilde v(x)\geq  \lambda _0+1\}.$ 
	Let $\eta:[\lambda_0+1,\infty)\to \mathbb{R}$ be  given by 
	$$
	\eta(\lambda)=\inf\{-|D\tilde v(x)|^{-1}\,\langle D\tilde v(x),\nu(S)(x)\rangle:x\in S\text{ and } \tilde v(x)\geq  \lambda \}.
	$$
	Clearly, $\eta\in C^{0,1}_{loc}([\lambda_0+1,\infty))$ and $\eta$ is nondecreasing. Since the intersection of $\Sigma_{\lambda_0+1}$ and $S$ is transverse, we see that $\eta(\lambda)>-1 $ for every $\lambda\geq \lambda _0+1$. Moreover, using \eqref{asymptotically flat section 3} and  \eqref{smoothly close}, we have $\eta(\lambda)\to1$ as $\lambda\to\infty$. Let $f\in C^\infty([\lambda_0+1,\infty))$ be a strictly increasing function such that  $f(\lambda)\to\infty$ as $\lambda\to\infty,$  
	$
	\tanh (f(\lambda))< \eta(\lambda)
	$ 
	for every $\lambda\geq \lambda_0+1$, 
and  $f(\lambda_0+1)\leq  t_0$. 
	The function $v=f\circ \tilde v$ has all  the asserted properties. 
\end{proof}

\begin{prop} \label{prop:super}
	Let $t\in\mathbb{R}$. Suppose that $\Sigma_t\subset \bar M(S)$ is a   stable minimal capillary surface with capillary angle $\arccos(-\tanh (t))$.   There is an  admissible surface $\tilde \Sigma_t\subset \bar M(S)$ disjoint from $\Sigma_t$ with $ \Omega (\Sigma_t)\subset \Omega(\tilde \Sigma_t)$ and a function  $w\in C^\infty(\bar\Omega(\tilde \Sigma_t)\setminus (\tilde \Sigma_t\cup \bar \Omega(\Sigma)))\cap C^0(\bar \Omega(\tilde \Sigma_t)\setminus \tilde \Sigma_t)$ with the following properties.  
\begin{equation} \label{w properties}
	\begin{aligned}
		&\circ \qquad w=t\text{ on } \bar \Omega(\Sigma_{t}),\\
		&\circ \qquad w(x)\to\infty \text{ as }x\to \tilde \Sigma_t,\\
		&\circ \qquad Dw\neq 0 \text{ in } \bar\Omega(\tilde \Sigma_t)\setminus (\tilde \Sigma_t\cup \bar \Omega(\Sigma_t)),\\
		&\circ \qquad \operatorname{div}(|Dw|^{-1}\,Dw)= 0\text{ in }\Omega(\tilde \Sigma_t)\setminus \bar \Omega(\Sigma_t),\text{ and} \\	
		&\circ \qquad |Dw|^{-1}\, \langle Dw ,\nu(S)\rangle >-\tanh (w)\text{ on } S(\tilde \Sigma_t)\setminus (\partial \tilde \Sigma_t\cup S(\Sigma_t)).	
	\end{aligned}
	\end{equation} 
%	Moreover, either  $Dw =0$ on $\Sigma_t$ or there exists a positive function $f\in C^\infty(\Sigma_t)$ with \begin{align} \label{speed2} 
	%	\begin{dcases}
	%		-\Delta^{\Sigma_t}  f -|h(\Sigma_t)|^2\,f =0\qquad&\text{ in }\Sigma_t\text{ and}
	%		\\-\langle \nabla ^{\Sigma_t} f ,\mu(\Sigma_t)\rangle -k(\partial \Sigma_t)\, f +(\cosh t)\,H(S)\, f =-(\cosh t)^{-1} &\text{ on }\partial \Sigma_t
	%	\end{dcases}
%	\end{align} 
%	such that  $|Dw|\leq 1/f$ on $\Sigma_t$. 
\end{prop}
\begin{proof}
	We first assume that $\Sigma_t$ is connected. Since $\Sigma_t$ is stable, Lemma \ref{local foliation} shows that there are $\delta>0$, $\omega\in C^\infty((-2\,\delta,2\,\delta))$, and  a smooth foliation $\{\Sigma(s):s\in(-2\,\delta,2\,\delta)\}$  of minimal capillary surfaces $\Sigma(s)$ with capillary angle $\arccos(-\tanh(\omega(s)))$ such that $\Sigma(0)=\Sigma_t$.
Let $\tilde \Sigma_t=\Sigma(\delta)$  and $\eta \in C^\infty([0,\delta))$ be such that 
\begin{itemize}
	\item[$\circ$] $\eta(s)\to\infty$ as $s\nearrow\delta$,
	\item[$\circ$] $\eta(s)>\omega(s)$ on $s\in(0,\delta)$, and
	\item[$\circ$] $\eta(0)=t$.
\end{itemize}
Then the function $w:\bar\Omega(\tilde \Sigma_t)\setminus \tilde \Sigma_t\to\mathbb{R}$ determined by $ w(x)=\eta(s)$ if $x\in \Sigma(s)$ for some $s\in(0,\delta)$ and $w(x)=0$ if $x\in \bar \Omega(\Sigma_t)$ has all of the asserted properties. This completes the proof in the case where $\Sigma_t$ is connected.

In the case where $\Sigma_t$ has multiple components $\Sigma^1_t,\ldots,\Sigma^\ell_t$,   we apply Lemma \ref{local foliation} to each  $\Sigma^i_t$ and obtain  $\delta^i>0$, $\omega^i\in C^\infty((-2\,\delta^i,2\,\delta^i))$, and  a smooth foliation $\{\Sigma^i(s):s\in(-2\,\delta^i,2\,\delta^i)\}$  by minimal capillary surfaces $\Sigma^i(s)$ with capillary angle $\arccos(-\tanh(\omega^i(s)))$ such that $\Sigma^i(0)=\Sigma^i_t$. Shrinking each $\delta^i>0$, we may assume that $\Sigma^1(\delta^1),\ldots, \Sigma^\ell(\delta^\ell)$ are mutually disjoint. We may now argue as in the case where $\Sigma_t$ is connected. The assertion follows.
\end{proof}

\section{Minimizing the free energy} 
In this section, we assume that $ S\subset \mathbb{R}^3$ is a properly embedded plane with integrable mean curvature. 
Moreover, we assume that the complement of a compact subset of $S$ is contained in the graph of a function $\psi \in C^\infty(\mathbb{R}^2)$ such that, for some $\tau>1/2$,
\begin{equation*} 
	\begin{aligned}
		\sum_{i=1}^2|\partial_i\psi(y)|+\sum_{i,\,j=1}^2|y|\,|\partial^2_{ij}\psi(y)|=O(|y|^{-\tau}).
	\end{aligned}
\end{equation*} 
Let $\nu(S)$ be the unit normal field of $S$ asymptotic to $-e_3$. The component of $\mathbb{R}^3\setminus S$ that $\nu(S)$ points out of is denoted by $M(S)$ and referred to as the region above $S$.  We denote the second fundamental form of $S$ with respect to $\nu(S)$ by $h(S)$. Our convention is such that the trace of $h(S)$, i.e., the mean curvature $H(S)$, is the tangential divergence of $\nu(S)$.

Recall from Appendix \ref{appendix:af surfaces} the definition of an admissible surface $\Sigma \subset \bar M(S)$ as well as that of its lateral surface $S(\Sigma)\subset S$ and its inside $\Omega(\Sigma)\subset M(S)$. 

Recall from Appendix \ref{appendix:free energy} the definition of a minimal capillary surface $\Sigma\subset \bar M(S)$ with capillary angle $\theta\in (0,\pi)$ and the notion of stability for such a surface. Moreover, given $t\in\mathbb{R}$,  recall  the definition \eqref{free energy smooth} of the free energy $J_t(\Sigma)$ with  angle $\arccos(-\tanh (t))$ of an admissible surface $\Sigma\subset \bar M(S)$.

Let $t_0\in\mathbb{R}$. Suppose that $\Sigma_{t_0}\subset \bar M(S)$ is a stable minimal capillary surface with capillary angle $\arccos(-\tanh (t_0))$. We assume that $H(S)\geq 0$ on $S\setminus S(\Sigma_{t_0})$.

The goal of this section is to prove that, for every $t\in (t_0,\infty),$ there is an admissible surface $\Sigma_t\subset \bar M(S)$ with $\Omega(\Sigma_{t_0})\subset \Omega(\Sigma_t)$ such that $J_t(\Sigma_t)\leq J_t(\Sigma)$ for every admissible surface $\Sigma\subset \bar M(S)$ with $\Omega(\Sigma_{t_0})\subset \Omega(\Sigma)$.

By Proposition \ref{prop:sub} and Lemma \ref{smoothing the corner}, there is a subsolution $v\in C^\infty(\bar M(S))$ satisfying \eqref{subsolution properties}.

By Proposition \ref{prop:super}, there is an admissible surface $\tilde\Sigma_{t_0}\subset \bar M(S)$ disjoint from $\Sigma_{t_0}$  with $ \Omega (\Sigma_{t_0})\subset \Omega(\tilde\Sigma_{t_0})$  and  a supersolution $w\in C^\infty(\bar\Omega(\tilde \Sigma_{t_0})\setminus (\tilde \Sigma_{t_0}\cup \bar \Omega(\Sigma_{t_0})))\cap C^0(\bar \Omega(\tilde \Sigma_{t_0})\setminus \tilde \Sigma_{t_0})$
satisfying \eqref{w properties} with $t=t_0$.  

Subtracting a constant from $v$, we may assume that  $v< w$ in $\bar \Omega(\tilde \Sigma_{t_0})\setminus \tilde \Sigma_{t_0}$. 

Given $t\in(t_0,\infty)$, let 
\begin{itemize} 
\item[$\circ$] $V_t=\{x\in \bar M(S):v(x)= t\}$ and 
\item[$\circ$] $W_t=\{x\in \bar M(S):w(x)= t\}.$
\end{itemize}  

Let $\Phi:\mathbb{R}\to\mathbb{R}$ be given by $\Phi(t)=\tanh(t)$. The relevant properties of $\Phi$ for the purpose of this paper are that
	\begin{align} 
		&\circ \qquad \Phi(0)=0, \label{Phi zero} \\
%		&\circ \qquad -1<\Phi(t)<1\text{ for every $t\in\mathbb{R}$}, \label{Phi derivative} \\
		&\circ \qquad \lim_{t\to-\infty} \Phi(t)=-1\text{ and} \lim_{t\to\infty} \Phi(t)=1\text{, and} \label{capillary angle pi}\\
			&\circ \qquad \Phi'(t)>0\text{ for every $t\in\mathbb{R}$}.\qquad\qquad\qquad\qquad\qquad\qquad\qquad \label{Phi convex}
	\end{align}

 Recall from Appendix \ref{BV}  the definition of a set $F\subset  M(S)$ of finite perimeter in $ M(S)$, of its reduced boundary $\partial^* F$ in $ M(S)$, and of its perimeter measure $||\partial F||$  in $M(S)$.  Moreover, recall that such a set  $F\subset  M(S)$ has a trace $\mathbbm{1}_F|_S\in L_{loc}^1(S)$.  
 
  Let $\mathcal{U}$ be the set of all bounded sets $F\subset M(S)$ of finite perimeter in $M(S)$ such that $\Omega(\Sigma_{t_0})\subset F$.  Given $t\in(t_0,\infty)$ and $F\in\mathcal{U}$,  let  
 \begin{align*} 
 	\tilde J_t(F)=||\partial F||(M(S))-\Phi(t)\,\int_{S}\mathbbm{1}_F|_{S}
 \end{align*} 
 be the   free energy $\tilde J_t(F)$ with angle $\arccos(-\Phi(t))$; cp.~\eqref{free energy smooth}. 
 Let 
 $$
 \Lambda_{t}=\inf \{\tilde J_t(F):F\in \mathcal{U}\}. 
 $$
  Note that $\Omega(V_t)\in \mathcal{U}$. In particular, $\Lambda_t<\infty$.
  
  \begin{lem}\label{freelowersemi} Let $K\subset \mathbb{R}^3$ be a compact set and $t\in(t_0,\infty)$.  Suppose that  $\{F_k\}_{k=1}^\infty$ is a sequence of sets $F_k\subset K\cap M(S)$ of finite perimeter in $M(S)$ such that $\mathbbm{1}_{F_k}\to \mathbbm{1}_F$ in $L^1(M(S))$ where $F\subset K\cap  M(S)$ is a set of finite perimeter in $M(S)$. There holds
  	$$
\tilde  	J_t(F)\leq \liminf_{k\to\infty} \tilde J_t(F_k).
  	$$
  \end{lem}
  \begin{proof}
  Given $\delta>0$, we choose a function $\eta_\delta \in C^\infty(\bar M(S))$ with the following properties:
  \begin{align*}
  	&\circ \qquad 0\leq \eta_\delta \leq 1,\qquad\qquad \qquad \circ \qquad |D\eta_\delta|\leq 2\,\delta^{-1},
  	\\&\circ\qquad  \eta_\delta =1 \text{ on $S$, and}\qquad \quad\,\,\, \circ \qquad \operatorname{spt}\eta_\delta \subset \{x\in \bar M(S):\operatorname{dist}(x,S)\leq \delta\}.
  \end{align*}
  By \eqref{capillary angle pi} and \eqref{Phi convex}, we have $|\Phi(t)|<1$. It follows that  
  $$
\Phi(t)  \int_{S}\mathbbm{1}_{F_k}|_S-\Phi(t)  \int_{S}\mathbbm{1}_{F}|_S\geq -\int_{S}\left|\mathbbm{1}_{F_k}|_S-\mathbbm{1}_{F}|_S\right |.
  $$
  Clearly,
  $$
 \int_{S}\left|\mathbbm{1}_{F_k}|_S-\mathbbm{1}_{F}|_S\right |= \int_{S}\eta_\delta\,\left|\mathbbm{1}_{F_k}|_S-\mathbbm{1}_{F}|_S\right |.
  $$
  By Lemma \ref{bv:trace1},  
  $$
  \int_{S}\eta_\delta\,\left|\mathbbm{1}_{F_k}|_S-\mathbbm{1}_{F}|_S\right|\leq \int_{M(S)}\,\eta_\delta\,\mathrm{d}\,||\partial F_k||+\int_{M(S)}\eta_\delta\, \mathrm{d}\,||\partial F||+\frac{2}{\delta}\,\int_{M(S)}|\mathbbm{1}_{F_k}-\mathbbm{1}_{F}|.
  $$
  Combining these inequalities, we have
  \begin{align*} 
J_t(F_k)-J_t(F)&\geq \int_{M(S)} (1-\eta_\delta)\,\mathrm{d}\,||\partial F_k||-\int_{M(S)} (1-\eta_\delta)\,\mathrm{d}\,||\partial F||
  	\\&\qquad -2\,\int_{M(S)} \eta_\delta\,\mathrm{d}\,||\partial F||-\frac{2}{\delta}\,\,\int_{M(S)}|\mathbbm{1}_{F_k}-\mathbbm{1}_F|.
  \end{align*} 
  Using Lemma \ref{bv:lower}, we see that  
  $$
  \liminf_{k\to\infty}\int_{M(S)}(1-\eta_\delta)\,\mathrm{d}\,||\partial F_k||\geq \int_{M(S)} (1-\eta_\delta)\,\mathrm{d}\,||\partial F||.
  $$
  Consequently,
  $$
  \liminf_{k\to\infty}\tilde J_t(F_k)\geq \tilde J_t(F) -2\,\int_{M(S)}\eta_\delta\,\mathrm{d}\,||\partial F||.
  $$
  Letting $\delta\searrow 0$, the assertion follows from the dominated convergence theorem.
  \end{proof}
 Recall from Lemma \ref{intersection and union} that if $E,\,F\subset M(S)$ are sets of finite perimeter in $M(S)$, then so are $E\cap F$ and $E\cup F$. 
  \begin{lem} \label{replacement lemma}
  	Let $t\in(t_0,\infty)$. Suppose that $F\in \mathcal{U}$. There holds 
  	\begin{align} \label{replace 1}
  		&\circ\qquad	\tilde	J_t( \Omega(V_t)\cap F)\leq \tilde J_t(F)\text{ and}\qquad\qquad\qquad\qquad\qquad\qquad \qquad \\
  		\label{replace 2}
  		&\circ\qquad  		\tilde J_t( \Omega(W_t)\cup F)\leq \tilde J_t(F).
  	\end{align}
  \end{lem}
  \begin{proof}
  	By Lemma \ref{bv:perimeterapprox}, there is a compact set $K\subset \mathbb{R}^3$ and a  sequence $\{F_k\}_{k=1}^\infty$ of  smooth compact sets $F_k\subset K$ such that $\mathbbm{1}_{M(S)\cap F_k}\to\mathbbm{1}_{F}$ in $L^1(M(S))$ and $||\partial (M(S)\cap F_k)||(M(S))\to ||\partial F||(M(S))$.  Using Lemma \ref{bv:trace2}, we see that 
  	\begin{align} \label{free approx} 
  		& 	\lim_{k\to\infty} \tilde J_t(M(S)\cap F_k)=\tilde J_t(F).  
  	\end{align}
  	We may assume that the intersection of $\partial F_k$ and $V_t$ and the intersection of $\partial F_k$ and $S$ are  transverse  for every $k$.
  	Let $E_k=\Omega(V_t)\cap F_k.$ Note that $E_k$ has Lipschitz boundary and that $E_k\subset K$.  Let $\nu(\partial F_k)$ be the normal of $\partial F_k$ pointing out of $F_k$ and $\nu (\partial E_k)$ be the normal of $\partial E_k$ pointing out of $E_k$. Note that  $\nu(\partial E_k)=|Dv|^{-1}\,Dv$   on $M(S)\cap (\partial E_k\setminus \partial F_k)$. Consequently,
  	\begin{align*} 
  &	|M(S)\cap \partial F_k|-|M(S)\cap \partial E_k|\\&\qquad \geq \int_{M(S)\cap(\partial F_k\setminus  \partial E_k)}|Dv|^{-1}\,\langle Dv,\nu(F_k)\rangle -\int_{M(S)\cap ( \partial E_k\setminus \partial F_k)}|Dv|^{-1}\,\langle Dv,\nu(\partial E_k)\rangle.
  	\end{align*} 
  	By the divergence theorem,  using \eqref{subsolution properties} and \eqref{Phi convex},
  	\begin{figure}\centering
  		\includegraphics[width=\linewidth]{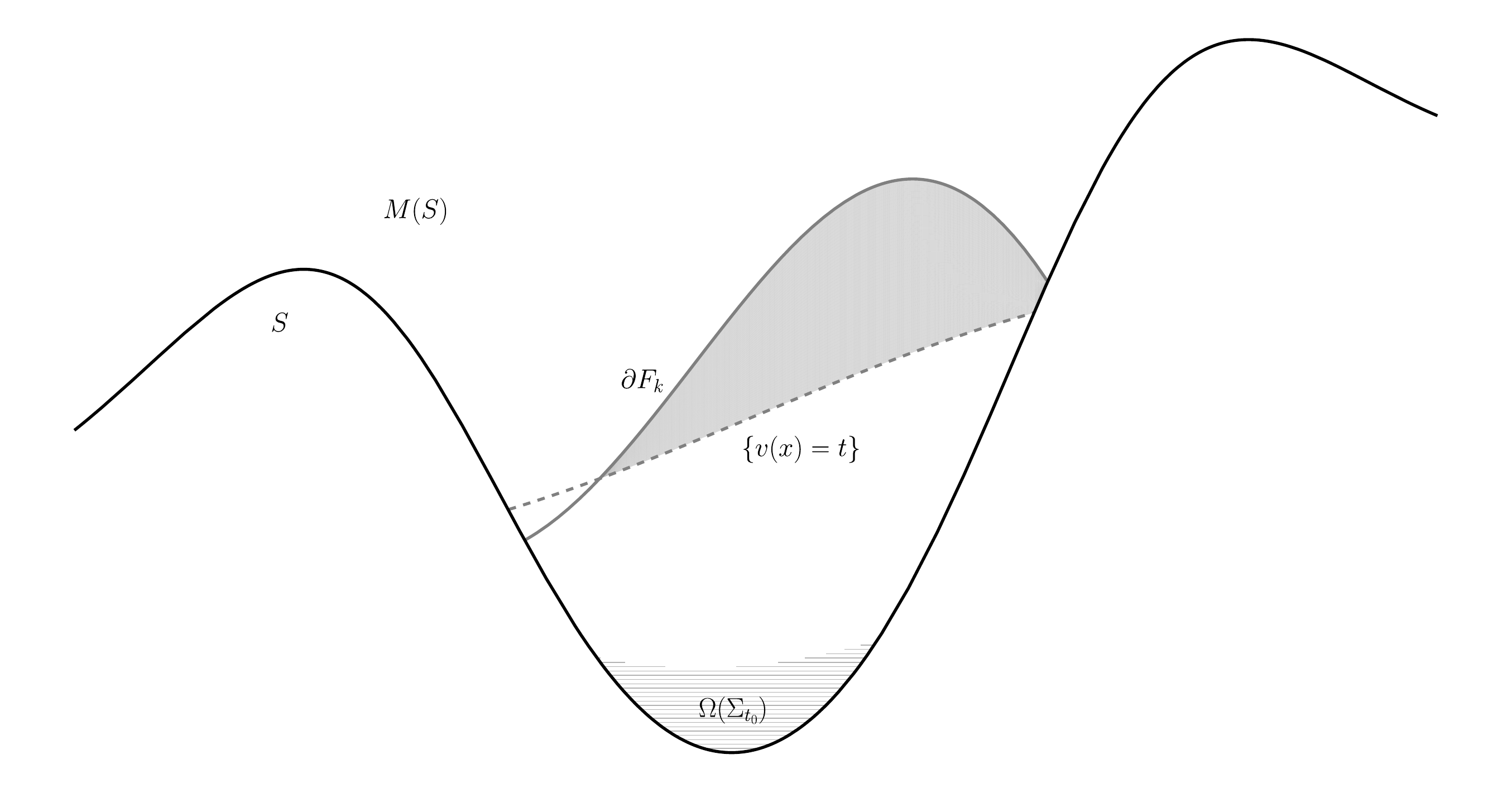}
  		\caption{An illustration of the calibration argument used in the proof of Lemma \ref{replacement lemma}. $S$ is indicated by the solid black line. $\Omega(\Sigma_{t_0})$ is indicated by the hatched region. $\partial F_k$ is indicated by the solid gray line and $\{x\in \bar M(S):v(x)=t\}$ is indicated by the dashed gray line. The divergence theorem is applied to the region indicated by the gray filling. 
  		}
  		\label{calibration}
  	\end{figure}
  	\begin{align*} 
  		&\int_{M(S)\cap(\partial F_k\setminus \partial E_k)}|Dv|^{-1}\,\langle Dv,\nu(\partial F_k)\rangle -\int_{M(S)\cap (\partial E_k\setminus \partial F_k)}|Dv|^{-1}\,\langle Dv,\nu(\partial E_k)\rangle
  		\\&\qquad = \int_{ M(S)\cap (F_k\setminus E_k)}\operatorname{div}(|D v|^{-1}\,D v)-\int_{S\cap(F_k\setminus E_k)}   |D v|^{-1}\,\langle D v,\nu(S)\rangle
  		\\&\qquad  \geq \int_{S\cap(F_k\setminus E_k)}  \Phi(v)
  		\\&\qquad \geq \Phi(t)\,|S\cap(F_k\setminus E_k)|;
  	\end{align*} 
  	see Figure \ref{calibration}.
It follows that
  	$$ \tilde J_t(M(S)\cap E_k)\leq \tilde J_t(M(S)\cap F_k).$$ Note that $\mathbbm{1}_{M(S)\cap E_k}\to \mathbbm{1}_{\Omega(V_t)\cap F}$ in $L^1(M(S)).$ Using  Lemma \ref{freelowersemi} and \eqref{free approx}, we obtain \eqref{replace 1}.
  	
   By the same reasoning as above, using \eqref{w properties} instead of \eqref{subsolution properties}, we obtain \eqref{replace 2}. The assertion follows. 
  \end{proof}
  \begin{prop} \label{capillary existence}
  	Let $t\in(t_0,\infty)$. There exists a  minimal capillary surface $\Sigma_t\subset \bar M(S)$ such that  $\Omega(\Sigma_{t_0})\subset \Omega(\Sigma_t)$ and 
  	\begin{align} \label{smooth inf} 
  \text{$J_t(\Sigma_t)\leq J_t(\Sigma)$   	for every admissible surface $\Sigma \subset \bar M(S)$ with $\Omega(\Sigma_{t_0})\subset \Omega(\Sigma)$.}
  	\end{align} 
  \end{prop}
  \begin{proof}
  	Let $\{F_k\}_{k=1}^\infty$ be a sequence of sets $F_k\in\mathcal{U}$  such that 
  	$$
  	\lim_{k\to\infty} \tilde J_t(F_k)=\Lambda_t.
  	$$ Let $E_k=(F_k\cap \Omega(V_t))\cup \Omega(W_t)$. By Lemma \ref{replacement lemma}, we have 
  	$$
 \tilde  	J_t(E_k)\leq \tilde J_t(F_k)
  	$$
  	for every $k\geq 1$. By Lemma \ref{bv:compact}, passing to a subsequence, there is a set $E\in\mathcal{U}$  such that $\mathbbm{1}_{E_k}\to\mathbbm{1}_{E}$ in $L^1(M(S))$.  By Lemma \ref{freelowersemi}, we have $\tilde J_t(E)=\Lambda_t$ so that 
  	\begin{align} \label{cacciopoli inf} 
  		\tilde J_t(E)\leq \tilde J_t(F)
  	\end{align}
  	for every $F\in \mathcal{U}$.  The pioneering work \cite[\S6]{Taylor} of J.~Taylor shows that   $\partial^*E$   is  smooth and that its closure $\Sigma_t$ is a   minimal capillary surface with capillary angle  $\arccos(-\Phi(t))$. One may also use the recent results in  \cite[Corollary 1.4]{DePhilippisMaggi} or \cite[Theorem 1.1]{OtisEdelenLi} in this step. Then \eqref{smooth inf} follows from \eqref{cacciopoli inf}. The assertion follows.
  \end{proof}

\section{Stable minimal capillary surfaces minimizing the free energy}
 \label{stable minimal section}

In this section, we assume that $ S\subset \mathbb{R}^3$ is a properly embedded plane with integrable mean curvature. 
Moreover, we assume that the complement of a compact subset of $S$ is contained in the graph of a function $\psi \in C^\infty(\mathbb{R}^2)$ such that, for some $\tau>1/2$,
\begin{equation} 
	\begin{aligned} \label{asymptotically flat section 4}
		\sum_{i=1}^2|\partial_i\psi(y)|+\sum_{i,\,j=1}^2|y|\,|\partial^2_{ij}\psi(y)|=O(|y|^{-\tau}).
	\end{aligned}
\end{equation} 
Let $\nu(S)$ be the unit normal field of $S$ asymptotic to $-e_3$. The component of $\mathbb{R}^3\setminus S$ that $\nu(S)$ points out of is denoted by $M(S)$ and referred to as the region above $S$.  We denote the second fundamental form of $S$ with respect to $\nu(S)$ by $h(S)$. Our convention is such that the trace of $h(S)$, i.e., the mean curvature $H(S)$, is the tangential divergence of $\nu(S)$.

Recall from Appendix \ref{appendix:af surfaces} the definition of an admissible surface $\Sigma \subset \bar M(S)$ as well as that of its lateral surface $S(\Sigma)\subset S$ and its inside $\Omega(\Sigma)\subset M(S)$. Moreover, recall the definition of the normal $\nu(\Sigma)$ and  the co-normal $\mu(\Sigma)$ as well as the conventions for the second fundamental form $h(\Sigma)$, the mean curvature $H(\Sigma)$, and the geodesic curvature $k(\partial \Sigma)$.

Recall from Appendix \ref{appendix:free energy} the definition of a minimal capillary surface $\Sigma\subset \bar M(S)$ with capillary angle $\theta\in (0,\pi)$ and the notion of stability for such a surface. Moreover, given $t\in\mathbb{R}$,  recall  the definition \eqref{free energy smooth} of the free energy $J_t(\Sigma)$ with  angle $\arccos(-\tanh (t))$ of an admissible surface $\Sigma\subset \bar M(S)$.

Let $t_0\in\mathbb{R}$. Suppose that $\Sigma_{t_0}\subset \bar M(S)$ is a stable minimal capillary surface with capillary angle $\arccos(-\tanh (t_0))$. We assume that $H(S)\geq 0$ on $S\setminus S(\Sigma_{t_0})$.

Given $t\in [t_0,\infty)$, let $\mathcal{E}_t$ be the set of all minimal capillary surfaces  $\Sigma_t\subset \bar M(S)$ such that  $\Omega(\Sigma_{t_0})\subset \Omega(\Sigma_t)$ and 
\begin{align} \label{mathE def}
J_t(\Sigma_t)\leq J_t(\Sigma) \text{ for every admissible surface $\Sigma\subset \bar M(S)$ with $\Omega(\Sigma_{t_0})\subset \Omega(\Sigma)$}.
\end{align} 
By Proposition \ref{capillary existence}, $\mathcal{E}_t\neq \emptyset$ for all $t\in(t_0,\infty)$. 

In this section, we study basic properties of the sets $\mathcal{E}_t$, $t\in[t_0,\infty)$.
\begin{lem}\label{sigmat is stable}
	Let $t\in[t_0,\infty)$ and suppose that $\Sigma_t\in\mathcal{E}_t$. Then $\Sigma_t$ is stable.
\end{lem}
\begin{proof}
	This follows from \eqref{mathE def} and Lemma \ref{stability}.
\end{proof}

\begin{lem} \label{capillary topology}
	Let $t\in[t_0,\infty)$ and suppose that $\Sigma_t\in \mathcal{E}_t$. Then  $\Sigma_t$ is diffeomorphic to a union of disks.
\end{lem}
\begin{proof}
	This follows from Lemma \ref{sigmat is stable} and Lemma \ref{minimization in homology class3}.
\end{proof}

By Proposition \ref{prop:sub} and Lemma \ref{smoothing the corner}, there is a subsolution $v\in C^\infty(\bar M(S))$ satisfying \eqref{subsolution properties}. 

By Proposition \ref{prop:super}, there is an admissible surface $\tilde\Sigma_{t_0}\subset \bar M(S)$ disjoint from $\Sigma_{t_0}$  with $ \Omega (\Sigma_{t_0})\subset \Omega(\tilde\Sigma_{t_0})$  and  a supersolution $w\in C^\infty(\bar\Omega(\tilde \Sigma_{t_0})\setminus (\tilde \Sigma_{t_0}\cup \bar \Omega(\Sigma_{t_0})))\cap C^0(\bar \Omega(\tilde \Sigma_{t_0})\setminus \tilde \Sigma_{t_0})$
satisfying \eqref{w properties} with $t=t_0$.   

Subtracting a constant from $v$, we may assume that  $v< w$ in $\bar \Omega(\tilde \Sigma_{t_0})\setminus \tilde \Sigma_{t_0}$. 

Given $t\in(t_0,\infty)$, let 
\begin{itemize} 
	\item[$\circ$] $V_t=\{x\in \bar M(S):v(x)= t\}$ and 
	\item[$\circ$] $W_t=\{x\in \bar M(S):w(x)= t\}.$
\end{itemize}

\begin{lem} \label{remain bounded}
	Let $t\in(t_0,\infty)$ and suppose that $\Sigma_t\in \mathcal{E}_t$. There holds $\Omega(W_t)\subset \Omega(\Sigma_t)\subset \Omega(V_t)$. 
\end{lem}
\begin{proof}
	By \eqref{subsolution properties}, for every $t'\in(t,\infty)$, $V_{t'}$ is an admissible surface that is minimal and such that  $\langle \nu(V_{t'}),\nu(S)\rangle < -\tanh (t')$ on $\partial V_{t'}$. Moreover $\{V_{t'}:t'\in(t,\infty)\}$ is a smooth foliation of the complement of a compact subset of $\bar M(S)$. By the strong maximum principle, we conclude that $\Omega(\Sigma_t)\subset \Omega(V_t)$; see, e.g., \cite[Lemma 1.13]{LiZhouZhu}. By the same reasoning, using \eqref{w properties} instead of \eqref{subsolution properties}, we conclude that $\Omega(W_t)\subset \Omega(\Sigma_t)$. The assertion follows. 
\end{proof}
\begin{lem} \label{compactness lemma}
	Let $\{t_k\}_{k=1}^\infty$ be a sequence of numbers $t_k\in[t_0,\infty)$ such that $t_k\to t$ for some $t\in[t_0,\infty)$. Suppose that $\Sigma_{t_k}\in \mathcal{E}_{t_k}$ for every $k$. There is $\Sigma_t\in \mathcal{E}_t$ such that, passing to a subsequence,  there holds $\Sigma_{t_k}\to \Sigma_t$  smoothly. 
\end{lem}
\begin{proof}
By		\eqref{asymptotically flat section 4}, the curvature of $S$ is bounded. Consequently,  	by \cite[Theorem 1.6]{HongSaturnino} or \cite[Theorem 1.12]{LiZhouZhu} and in view of \eqref{mathE def}, $|h(\Sigma_{t_k})|_{L^\infty(\Sigma_{t_k})}$ is bounded in terms of $|t_k|$. By \eqref{geodesic}, using again that the curvature of $S$ is bounded, $|k(\partial\Sigma_{t_k})|_{L^\infty(\partial \Sigma_{t_k})}$ is bounded in terms of $|t_k|$ and $|h(\Sigma_{t_k})|_{L^\infty(\partial \Sigma_{t_k})}$.  
By \cite[Lemma B.3]{LiZhouZhu}, $|\partial\Sigma_{t_k}|$ is bounded in terms of $|t_k|$ and $|\Sigma_{t_k}|$. By \eqref{mathE def} and Lemma \ref{remain bounded},
$$
\sup\{|\Sigma_{t_k}|:k\geq 1\}<\infty. 
$$
By a standard compactness argument for minimal surfaces using Lemma \ref{remain bounded}, we see that, passing to a subsequence,  $\Sigma_{t_k}\to \Sigma_t$  smoothly for some  minimal capillary surface $\Sigma_t$ with capillary angle $\arccos(-\tanh (t))$ such that $\Omega(\Sigma_{t_0})\subset \Omega(\Sigma_t)$; see \cite[p.~390-391]{choischoen} and \cite[\S6]{FraserLi}. It follows that $\Sigma_t$ satisfies \eqref{mathE def}. The assertion follows. 
\end{proof}
In particular, $\mathcal{E}_{t_0}\neq \emptyset$.  For convenient reference, we record the following elementary lemma.
\begin{lem} \label{sharp initial}
	Suppose that 
	$$
	\text{$J_{t_0}(\Sigma_{t_0})\leq J_{t_0}(\Sigma)$ for every admissible surface $\Sigma\subset \bar M(S)$ with $\Omega(\Sigma_{t_0})\subset \Omega(\Sigma)$}
	$$
	with equality if and only if $\Sigma=\Sigma_{t_0}$.  There holds  $\mathcal{E}_{t_0}=\{\Sigma_{t_0}\}$.
\end{lem}
Given $t\in[t_0,\infty)$, let
\begin{equation} \label{s def}
\begin{aligned}
	&\circ\qquad  s_t^-=\inf \{|S(\Sigma_t)|:\Sigma_t\in \mathcal{E}_t\}\text{  and}\hspace{6cm} \\
	&\circ \qquad s_t^+=\sup \{|S(\Sigma_t)|:\Sigma_t\in \mathcal{E}_t\}.
\end{aligned}
\end{equation} 

\begin{prop} \label{universal existence}
	Let $t\in [t_0,\infty)$. There are $\Sigma_t^-,\,\Sigma_t^+\in \mathcal{E}_t$ such that $|S(\Sigma_t^-)|=s_t^-$ and $|S(\Sigma_t^+)|=s_t^+$. Moreover, 
	\begin{align}\label{interpolate}
		|\Sigma^+_t|=|\Sigma^-_t|+\Phi(t)\,(s_t^+-s_t^-).
	\end{align}
\end{prop}
\begin{proof}
	By Proposition \ref{capillary existence}, $\mathcal{E}_t\neq \emptyset$  for every $t\in(t_0,\infty)$. Using Lemma \ref{compactness lemma}, we see that $\mathcal{E}_{t_0}\neq\emptyset$. Let $t\in[t_0,\infty)$. Using Lemma \ref{compactness lemma} again, we see that there are $\Sigma_t^-,\,\Sigma_t^+\in \mathcal{E}_t$ such that $|S(\Sigma_t^-)|=s_t^-$ and $|S(\Sigma_t^+)|=s_t^+$. By \eqref{mathE def}, we have $J_t(\Sigma^+_t)=J_t(\Sigma^-_t)$. Thus \eqref{interpolate} holds, as asserted.
\end{proof}

\begin{lem} \label{s property 1}
	Let $t_1,\,t_2\in [t_0,\infty)$ be such that $t_2>t_1$. There holds $s^-_{t_2}>s^+_{t_1}$.
\end{lem}
\begin{proof} Let $\Sigma_{t_1}\in \mathcal{E}_{t_1}$ and $\Sigma_{t_2}\in \mathcal{E}_{t_2}$. Note that $\Sigma_{t_1}\notin \mathcal{E}_{t_2}$ and $\Sigma_{t_2}\notin \mathcal{E}_{t_1}$. By \eqref{mathE def},  $J_{t_1}(\Sigma_{t_1})<J_{t_1}(\Sigma_{t_2})$ and $J_{t_2}(\Sigma_{t_2})<J_{t_2}(\Sigma_{t_1})$. Thus,
	$$
	(\Phi(t_2)-\Phi(t_1))\,|S(\Sigma_{t_1})|<(\Phi(t_2)-\Phi(t_1))\,|S(\Sigma_{t_2})|.
	$$
	Using \eqref{Phi convex}, the assertion follows. 
\end{proof}
\begin{lem} \label{s property 2}
	Let $t\in [t_0,\infty)$ and $\{t_k\}_{k=1}^\infty$ be a sequence of numbers $t_k\in[t_0,\infty)$. Assume that $t_k\searrow t$. There holds
	\begin{itemize}
		\item[$\circ$] $s^-_{t_k}\searrow s^+_t$ and
		\item[$\circ$] $s^+_{t_k}\searrow s^+_t$.
	\end{itemize}
	Assume that $t_k\nearrow t$. There holds
	\begin{itemize}
	\item[$\circ$] $s^-_{t_k}\nearrow s^-_t$ and
	\item[$\circ$] $s^+_{t_k}\nearrow s^-_t$.
\end{itemize}
\end{lem}
\begin{proof}
	This follows from Lemma \ref{compactness lemma} and Lemma \ref{s property 1}.
\end{proof}

\begin{lem} \label{inner divergence}
	Let $K\subset \bar M(S)$ be a compact set. For every sufficiently large $t\in[t_0,\infty)$, there holds $K\cap \Sigma_t=\emptyset$ for all $\Sigma_t\in \mathcal{E}_t$. 
\end{lem}
\begin{proof}
	Suppose, for a contradiction, that the assertion of the lemma is false. Then there exists a sequence $\{t_k\}_{k=1}^\infty$ of numbers $t_k\in [t_0,\infty)$ with  $t_k\to\infty$  and surfaces $\Sigma_{t_k}\in \mathcal{E}_{t_k}$  such that $K\cap\Sigma_{t_k}\neq \emptyset$ for every $k$. 
	
	By assumption, $S$ is diffeomorphic to the plane. By Lemma \ref{minimization in homology class3}, $\Sigma_{t_0}$ is diffeomorphic to a union of disks. It follows that there are $p_1,\ldots,p_\ell \in \mathbb{R}^2$ with $|p_i-p_j|>2$ for all $1\leq i,\,j\leq \ell$ with $i\neq j$ and a diffeomorphism
	$$
	\Psi:\mathbb{R}^2\setminus \bigcup_{i=1}^\ell \{y\in \mathbb{R}^2:|y-p_i|<1\}\to (S\setminus S(\Sigma_{t_0}))\cup \partial \Sigma_{t_0}. 
	$$
	We may assume that $p_1=0$. Since $\Omega(\Sigma_{t_0})\subset \Omega(\Sigma_{t_k})$, $\Psi^{-1}(\partial \Sigma_{t_k})$ encloses $\{y\in \mathbb{R}^2:|y|<1\}$ for every $k$. Let $r_k=\min\{|y|:y\in \Psi^{-1}(\partial \Sigma_{t_k})\}$ and note that, by Lemma \ref{remain bounded}, $\liminf_{k\to\infty} r_k>1$. Let $y_k\in \Psi^{-1}(\partial \Sigma_{t_k})$ be such that $|y_k|=r_k$. 
		
	Assume first that $r_k=O(1)$. It follows that there are Jordan curves $$\Gamma_k\subset \mathbb{R}^2\setminus \bigcup_{i=1}^\ell \{y\in \mathbb{R}^2:|y-p_i|\leq 1\}$$ with the following properties; see Figure \ref{curves}.
		\begin{figure}\centering
		\includegraphics[width=\linewidth]{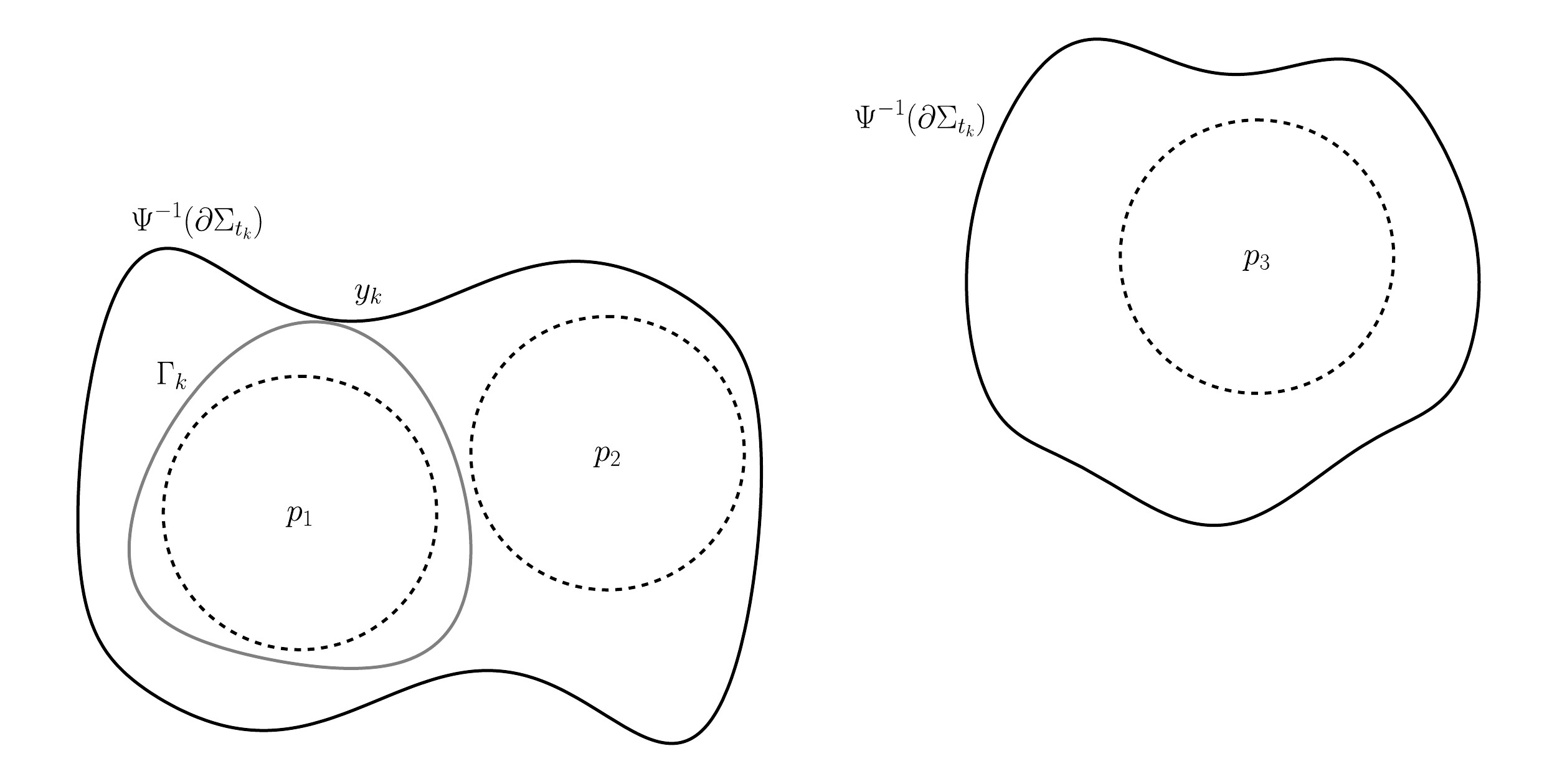}
		\caption{An illustration of the proof of Lemma \ref{inner divergence}. $\Psi^{-1}(\partial \Sigma_{t_0})$ has three components, which are illustrated by the dashed black lines. $\Psi^{-1}(\partial \Sigma_{t_k})$ has two components, which are illustrated by the solid black lines. $\Gamma_k$, which is illustrated by the solid gray line, touches one of these two components at $y_k$.  
		}
		\label{curves}
	\end{figure}
	\begin{itemize}
		\item[$\circ$] $\Gamma_k$ is homotopic to $\{y\in \mathbb{R}^2:|y|=1\}$ in $\mathbb{R}^2\setminus \bigcup_{i=1}^\ell \{y\in \mathbb{R}^2:|y-p_i|<1\}$. 
		\item[$\circ$] There holds $\Gamma_k\cap \Psi^{-1}(\partial \Sigma_{t_k})=\{y_k\}$. 
		\item[$\circ$] There holds $\sup\{|y|:y\in  \Gamma_k\}=r_k.$
		\item[$\circ$] There holds $	\liminf_{k\to\infty}\inf\{|y|:y\in  \Gamma_k\}>1.$
				\item[$\circ$] $\Gamma_k$ converges smoothly to a Jordan curve $\Gamma$ as $k \to \infty$.
	\end{itemize}
	Using that $H(S)\geq 0$ on $S\setminus S(\Sigma_{t_0})$ and a calibration argument based on the level sets of $w$ as in the proof of Lemma \ref{replacement lemma}, we see that there is a connected admissible surface $\tilde \Sigma_k\subset \bar M(S)\setminus \bar \Omega  (\Sigma_{t_0})$ with boundary $\partial \tilde \Sigma_k = \Psi (\Gamma_k)$ such that
	\begin {itemize}
	\item [$\circ$] 
	$|\tilde \Sigma_k|\leq |\Sigma|$ for every admissible surface $\Sigma \subset \bar M(S)\setminus \bar \Omega(\Sigma_{t_0})$ with $\partial \Sigma=\Psi(\Gamma_k)$ and 
	\item [$\circ$]  $\liminf_{k\to\infty}\operatorname{dist}(\tilde \Sigma_k,\Sigma_{t_0})>0$.
	\end {itemize}
	We claim that 
	\begin{align} \label{not tangent} 
		\liminf_{k\to\infty}\inf\{\langle \nu(\tilde \Sigma_k)(x),\nu(S)(x)\rangle:x\in \partial \tilde \Sigma_k\}>-1.
	\end{align} 
	If not, using  standard tools from minimal surface theory, we see that $\tilde \Sigma_k$ converges smoothly to an admissible minimal surface $\tilde \Sigma\subset \bar M(S)\setminus \bar \Omega(\Sigma_{t_0})$ with $\partial \tilde \Sigma = \Psi (\Gamma)$ that is tangent to $S$ somewhere on $\tilde \Sigma$. Then  $\tilde \Sigma\subset S$, by the strong maximum principle, which is impossible.
	
	 In view of \eqref{not tangent}, using that, by \eqref{capillary angle pi}, the capillary angle of $\Sigma_{t_k}$ converges to $\pi$ as $k\to\infty$,  we see that $\Sigma_{t_k}$ and $\Sigma_k$ intersect transversely at   $\Psi(y_k)$ provided that $k$ is sufficiently large. As $\partial \Sigma_{t_k}\cap \partial \tilde \Sigma_k=\{\Psi(y_k)\},$ it follows that  the admissible Lipschitz surface $\hat \Sigma_{k}\subset \bar M(S)$ with $\Omega(\hat \Sigma_{t_k})=\Omega(\Sigma_{t_k})\cap \Omega(\tilde \Sigma_k)$ satisfies $\partial\hat \Sigma_k=\Psi(\Gamma_k)$ and $\hat \Sigma_k\subset \bar M(S)\setminus \bar \Omega(\Sigma_{t_0})$. Moreover, using \eqref{mathE def}, we see that $|\hat \Sigma_k|=|\tilde \Sigma_k|$. Since $\hat \Sigma_{k}$ has an edge, we see that $\tilde \Sigma_k$ is not a least area surface in $\bar M(S)\setminus  \bar \Omega(\Sigma_{t_0})$, a contradiction.

	We now consider the case where $r_k\to\infty$. Using \eqref{mathE def}, standard tools from minimal surface theory, and that $K\cap \Sigma_{t_k}\neq \emptyset$ for every $k$, we see that, passing to a subsequence, $\Sigma_{t_k}$ converges locally smoothly to an unbounded area-minimizing boundary contained in $ \bar M(S)\setminus \bar \Omega(\Sigma_{t_0})$. As as in the proof of the half-space theorem  \cite[Theorem 1]{HoffmanMeeks} as we did in the proof of Lemma \ref{subsolutions}, we obtain a contradiction.
	
	The assertion follows. 
\end{proof}

\section{Monotonicity of the free energy mass}

In this section, we assume that $ S\subset \mathbb{R}^3$ is a properly embedded plane with integrable mean curvature. 
Moreover, we assume that the complement of a compact subset of $S$ is contained in the graph of a function $\psi \in C^\infty(\mathbb{R}^2)$ such that, for some $\tau>1/2$,
\begin{equation*} 
	\begin{aligned}
		\sum_{i=1}^2|\partial_i\psi(y)|+\sum_{i,\,j=1}^2|y|\,|\partial^2_{ij}\psi(y)|=O(|y|^{-\tau}).
	\end{aligned}
\end{equation*} 
Let $\nu(S)$ be the unit normal field of $S$ asymptotic to $-e_3$. The component of $\mathbb{R}^3\setminus S$ that $\nu(S)$ points out of is denoted by $M(S)$ and referred to as the region above $S$.   We denote the second fundamental form of $S$ with respect to $\nu(S)$ by $h(S)$. Our convention is such that the trace of $h(S)$, i.e., the mean curvature $H(S)$, is the tangential divergence of $\nu(S)$.

Recall from Appendix \ref{appendix:af surfaces} the definition of an admissible surface $\Sigma \subset \bar M(S)$ as well as that of its lateral surface $S(\Sigma)\subset S$ and  its inside $\Omega(\Sigma)\subset M(S)$. Moreover, recall the definition of the normal $\nu(\Sigma)$, the conventions for the second fundamental form $h(\Sigma)$ and mean curvature $H(\Sigma)$, as well as the definitions of the Gauss curvature $K(\Sigma)$ and the Euler characteristic $\chi(\Sigma)$.

	Recall from Appendix \ref{appendix:variation} the definition of an admissible variation $\{\Sigma(s):s\in(-\varepsilon,\varepsilon)\}$ of an admissible surface $\Sigma \subset \bar M(S)$. 

Recall from Appendix \ref{appendix:free energy} the definition of a minimal capillary surface $\Sigma\subset \bar M(S)$ with capillary angle $\theta\in (0,\pi)$ and the notion of stability for such a surface. Given $t\in\mathbb{R}$, recall the definition   \eqref{free energy smooth} of the free energy $J_t(\Sigma)$ with angle $\arccos(-\tanh (t))$ of an admissible surface $\Sigma \subset \bar M(S)$. Moreover, recall the definition \eqref{free energy mass} of the  free energy mass $m_f(\Sigma)$ of a minimal capillary surface $\Sigma\subset \bar M(S)$.

Let $t_0\geq 0$. Suppose that $\Sigma_{t_0}\subset \bar M(S)$ is a stable  minimal capillary surface with capillary angle $\arccos(-\tanh(t_0))$. We assume that $H(S)\geq 0$ on   $S\setminus S(\Sigma_{t_0})$.

Given $t\in[t_0,\infty)$, recall from Section \ref{stable minimal section} the definition of the sets $\mathcal{E}_{t}$ of all minimal capillary surfaces  $\Sigma_t\in \mathcal{E}_t$ satisfying \eqref{mathE def}. Recall from \eqref{s def} the definitions of $s^-_t,\,s^+_t$ and that, according to Proposition \ref{universal existence}, there are $\Sigma^-_t,\,\Sigma^+_t\in\mathcal{E}_t$ with $|S(\Sigma^-_t)|=s^-_t$ and $|S(\Sigma^+_t)|=s^+_t$.

The goal of this section is to prove that the function    $t\mapsto m_f(\Sigma^-_t)$ is nondecreasing.

Recall that $\Phi\in C^\infty(\mathbb{R})$ is given by $\Phi(t)=\tanh(t)$.

In view of Lemma \ref{s property 1} and Lemma \ref{s property 2}, we may define a continuous and nondecreasing  function $\sigma:(|S(\Sigma^+_{t_0})|,\infty)\to (t_0,\infty)$ that is determined by $\sigma(s)=t$ where $t\in(t_0,\infty)$ is such that $s\in[s^-_t,s^+_t]$. 
By Lemma \ref{inner divergence}, $\sigma$ is surjective. Let $\upsilon:(|S(\Sigma^{+}_{t_0})|,\infty)\to\mathbb{R}$ be given by $$  \upsilon(s)=
	|\Sigma^-_{\sigma(s)}|+\Phi(\sigma(s))\,(s-s^-_{\sigma(s)}).$$
%	Note that $\upsilon(s)=J_{\sigma(s)}(\Sigma^-_{\sigma(s)})+s$. 
	\begin{rema}
		Assume that $t\in(t_0,\infty)$ is such that $s^-_t=s^+_t=s$. Then $\sigma(s)=t$ and $\upsilon(s)=|\Sigma_t|$ for every $\Sigma_t\in\mathcal{E}_t$. 
	\end{rema}
Using Lemma \ref{compactness lemma}, Lemma \ref{s property 2}, \eqref{interpolate}, and that $\sigma$ is continuous, we see that $\upsilon$ is continuous.

According to	 Lemma \ref{rho k lower bound} below,  $\upsilon$ is a lower bound for the free energy profile \eqref{free energy profile intro}.
\begin{lem} \label{rho k lower bound}
	Let $\Sigma \subset \bar M(S)$ be an  admissible surface with $\Omega(\Sigma_{t_0})\subset \Omega(\Sigma)$ and $|S(\Sigma_{t_0}^+)|< |S(\Sigma)|$. There holds 
	$$
	\upsilon(|S(\Sigma)|)\leq |\Sigma|.
	$$
	\end{lem}
\begin{proof}
 By \eqref{mathE def}, we have
$
J_{\sigma(|S(\Sigma)|)}(\Sigma^-_{\sigma (|S(\Sigma)|)})\leq J_{\sigma(|S(\Sigma)|)}(\Sigma).
$
Using that $s^-_{\sigma(|S(\Sigma)|)}=|S(\Sigma^-_{\sigma(|S(\Sigma))})|$, we conclude that 
$$
\upsilon(|S(\Sigma)|)=|\Sigma^-_{\sigma(|S(\Sigma)|)}|+\Phi(\sigma(|S(\Sigma)|))\,(|S(\Sigma)|-s^-_{\sigma(|S(\Sigma)|)})\leq |\Sigma|,
$$
as asserted.
\end{proof}
The left-derivative and right-derivative   of a nondecreasing function $f$ are denoted by $\ubar f'$ and $\bar f'$, respectively. 
\begin{lem} \label{rho k increasing}
The function	$ \upsilon:(|S(\Sigma^{+}_{t_0})|,\infty)\to\mathbb{R}$
 is strictly increasing. Moreover, there holds 
	\begin{align} \label{ld and rd} 
		\bar {\upsilon}'(s)\leq \Phi(\sigma(s))\leq \ubar\upsilon'(s)
	\end{align} 
	for every $s\in (|S(\Sigma^{+}_{t_0})|,\infty)$.
\end{lem}
\begin{proof}
	Let $s_1,\,s_2\in (|S(\Sigma^{+}_{t_0})|,\infty)$ be such that $s_2> s_1$. 
	
	To prove that $\upsilon$ is strictly increasing, using that, by \eqref{Phi convex} and \eqref{Phi zero}, $\Phi(t)> 0$ for every $t>0$, we may assume  
	that  $\sigma(s_2)>\sigma(s_1)$. Then, using that $\Phi(t)> 0$ for every $t> 0$  and \eqref{interpolate}, we have $\upsilon(s_1)\leq |\Sigma^{+}_{\sigma(s_1)}|$ and $\upsilon(s_2)\geq  |\Sigma^-_{\sigma(s_2)}|$. According to \eqref{mathE def}, $J_{\sigma(s_1)}(\Sigma^{+}_{\sigma(s_1)})\leq J_{\sigma(s_1) }(\Sigma^-_{\sigma(s_2)})$. Using that $\Phi(t)>0$ for every $t>0$ and that  $|S(\Sigma^{+}_{\sigma(s_1)})|<|S(\Sigma^-_{\sigma(s_2)})|$ by Lemma \ref{s property 1}, we obtain that $|\Sigma^{+}_{\sigma(s_1)}|<|\Sigma^-_{\sigma(s_2)}|$ and conclude that $\upsilon(s_2)> \upsilon(s_1)$, as asserted.  
	
Assume that $s_1< s^+_{\sigma(s_1)}$. Note that $\sigma(s_1)=\sigma(s_2)$ provided that $s_2-s_1>0$ is sufficiently small. It follows that 
	$$
	\upsilon(s_2)- \upsilon(s_1)=\Phi(\sigma(s_1))\,(s_2-s_1).
	$$ 
	Assume that $s_1=s^+_{\sigma(s_1)}$. By \eqref{interpolate}, $\upsilon(s_1)=|\Sigma^+_{\sigma(s_1}|$.  According to \eqref{mathE def},  $J_{\sigma(s_2)}(\Sigma^-_{\sigma(s_2)})\leq J_{\sigma(s_2)}(\Sigma^+_{\sigma(s_1)})$. 	Consequently,
	\begin{align*}
		\upsilon(s_2)-\upsilon(s_1)&\leq  \Phi(\sigma(s_2))\,(s_2-s_1).
	\end{align*}
	Using that $\sigma$ is continuous, we conclude that $\bar \upsilon'(s_1)\leq \Phi(\sigma(s_1))$. A similar argument shows that $\Phi(\sigma(s_1))\leq \ubar  \upsilon'(s_1).$ The assertion follows.
\end{proof}
The following lemma and its proof should be compared to the arguments in \cite[\S7]{BavardPansu} and  the proof of  \cite[Theorem 3]{bray}.

\begin{lem} \label{comparison function}
The following property holds for every $\delta>0$ sufficiently small. 
	For every $s\in (|S(\Sigma^{+}_{t_0})|+\delta,\infty)$, there exists $\tilde \delta\in(0,\delta)$ and $\tilde \upsilon \in C^\infty((s-\tilde \delta,s+\tilde \delta))$ such that 
	\begin{itemize}
		\item[$\circ$] $\tilde \upsilon(s)= \upsilon(s)$,
		\item[$\circ$] $\tilde\upsilon(s')\geq \upsilon(s')$ for every $s'\in(s-\tilde \delta,s+\tilde \delta)$,
		\item[$\circ$] $\tilde \upsilon'(s)=\tanh(\sigma(s))$, and 

		\item[$\circ$] $2\,\tilde \upsilon''(s)\,\tilde \upsilon(s)\leq 1-\tilde \upsilon'(s)^2$  .

	\end{itemize}
Moreover, $\tilde \upsilon \in C^\infty((s-\tilde \delta, s+\tilde \delta))$ can be chosen such that, in addition, 	\begin{align} \label{strict} 2\,\tilde \upsilon''(s)\,\tilde \upsilon(s)< 1-\tilde \upsilon'(s)^2\end{align} unless 
\begin{itemize}
	\item[$\circ$] $s=s^-_{\sigma(s)}$ or $s=s^+_{\sigma(s)}$,
	\item[$\circ$] every $\Sigma_{\sigma(s)}\in \mathcal{E}_{\sigma(s)}$ with $|\Sigma_{\sigma(s)}|= s$ is a union of flat disks, and
	\item[$\circ$]   $H(S)=0$ on $\partial \Sigma_{\sigma(s)}$ for every   $\Sigma_{\sigma(s)}\in \mathcal{E}_{\sigma(s)}$.
\end{itemize}
\end{lem}
\begin{proof}
	Let $\delta>0$ and $s\in (|S(\Sigma^{+}_{t_0})|+\delta,\infty)$.
	
	In the case where  $s\in (s^-_{\sigma(s)},s^+_{\sigma(s)})$, using that $\upsilon$ is locally bounded   and linear near $s,$ we see that the function $\tilde \upsilon \in C^\infty((s-\delta/2,s+\delta/2))$ defined by $$\tilde \upsilon(s')=\upsilon(s)+\Phi(\sigma(s))\,(s'-s)+c(s)\,(s'-s)^4$$ satisfies $\tilde \upsilon(s)=\upsilon(s)$ and $\tilde \upsilon(s')\geq \upsilon(s')$ for every $s'\in(s-\delta/2,s+\delta/2)$ provided that $c(s)>1$ is sufficiently large. Note that, in this case, $\tilde \upsilon''(s)=0$ and $\tilde \upsilon'(s)=\tanh (\sigma(s))$, so that \eqref{strict} holds.  
	
	 In the case where  $s=s^-_{\sigma(s)}$ or $s=s^+_{\sigma(s)}$, let $\Sigma_{\sigma(s)}\in \mathcal{E}_{\sigma(s)}$ be such that $|\Sigma_{\sigma(s)}|=s$ and  $\Sigma^1_{\sigma(s)},\ldots ,\Sigma^\ell_{\sigma(s)}$ be the components of $\Sigma_{\sigma(s)}$. We apply Lemma \ref{normal variation} to $\Sigma_{\sigma(s)}$ with $\Sigma=\Sigma_{t_0}$ and
	 $$
	 \gamma_i=\frac{|\partial \Sigma^i_{\sigma(s)}|}{\sum_{j=1}^\ell|\partial \Sigma^j_{\sigma(s)}|^2},
	 $$
	   $1\leq i\leq \ell$. This gives $\tilde \delta\in(0,\delta)$,  a function $\tilde \upsilon \in C^\infty((s-\tilde \delta,s+\tilde \delta)),$ and an admissible variation $\{\Sigma(s'):s'\in(s-\tilde \delta,s+\tilde \delta)\}$ of $\Sigma_{\sigma(s)}$  with the following properties. 
	 \begin{align*}
%	 	&\circ \qquad \Sigma(s)=\Sigma_{\sigma(s)},\qquad \qquad \qquad \qquad\qquad\qquad\qquad\qquad\qquad\qquad\qquad \\
	 	&\circ \qquad \tilde \upsilon'(s)=\tanh (\sigma(s)), \qquad \qquad \qquad\qquad\qquad\qquad\qquad\qquad\qquad\qquad
	 		\\&\circ \qquad \tilde \upsilon(s')=|\Sigma(s')|\text{ for every }s'\in(s-\tilde \delta,s+\tilde\delta),
	 			\\&\circ \qquad  |S(\Sigma(s'))|=s'\text{ for every }s'\in(s-\tilde \delta,s+\tilde \delta),\text{ and}\\
&\circ \qquad \Omega(\Sigma_{t_0})\subset \Omega(\Sigma(s'))\text{ for every } s'\in(s-\tilde \delta,s+\tilde \delta).
	 \end{align*}
	Moreover, there holds 
	\begin{align*} 
		\tilde \upsilon''(s)&=\bigg(\sum_{i=1}^\ell\gamma_i\,|\partial \Sigma^i_{\sigma(s)}|\bigg)^{-2}\,(1-\tilde \upsilon'(|S(\Sigma_{\sigma(s)})|)^2)
		\\
		&\qquad \times \sum_{i=1}^\ell \gamma_i^2\left(2\,\pi\,\chi(\Sigma^i_{\sigma(s)})-\frac12\,\int_{\Sigma^i_{\sigma(s)}} |h(\Sigma_{\sigma(s)})|^2-\cosh (\sigma(s))\,\int_{\partial \Sigma^i_{\sigma(s)} } H(S)\right).
	\end{align*}
	Clearly, $\tilde \upsilon(s)=\upsilon(s)$. 
	By Lemma \ref{rho k lower bound}, there holds $\tilde \upsilon(s')\geq \upsilon(s')$ for every $s'\in(s-\tilde \delta,s+\tilde \delta)$. According to Lemma \ref{capillary topology}, $\Sigma_{\sigma(s)}$ is diffeomorphic to a union of disks so that $\chi(\Sigma^i_{\sigma(s)})=1$ for every $1\leq i\leq \ell$. 	Using that $H(S)\geq 0$ in $S\setminus S(\Sigma_{t_0})$, we have
	$$ 
	\int_{\partial \Sigma^i_{\sigma(s)}} H(S)\geq 0
	$$
for every $1\leq i\leq \ell$. Clearly, 
	$$
	\int_{\Sigma^i_{\sigma(s)}} |h(\Sigma_{\sigma (s)})|^2\geq 0.
	$$
 Using that 
$$
\sum_{i=1}^\ell \gamma_i\,|\partial \Sigma^i_{\sigma(s)}|=1\qquad\text{and}\qquad \sum_{i=1}^\ell \gamma_i^2=\bigg(\sum_{i=1}^\ell \,|\partial \Sigma^i_{\sigma(s)}|^2\bigg)^{-1},
$$
we obtain that 
$$
\tilde \upsilon''(s)\leq 2\,\pi\,\bigg(\sum_{i=1}^\ell \,|\partial \Sigma^i_{\sigma(s)}|^2\bigg)^{-1}\,(1-\tilde \upsilon'(|S(\Sigma_{\sigma(s)})|)^2).
$$
Finally, by the optimal isoperimetric inequality for minimal disks, 
	$$
	|\partial \Sigma^i_{\sigma(s)}|^2\geq 4\,\pi\,|\Sigma^i_{\sigma(s)}|
	$$
 for every $1\leq i\leq \ell$	with equality if and only if $\Sigma^i_{\sigma(s)}$ is a flat disk; see \cite[p.~1]{Carleman}. 
The assertion follows from these estimates.
\end{proof}
\begin{prop} \label{montonicity prop}
	Let $t_1,\,t_2\in (t_0,\infty)$ be such that $t_2>t_1$. There holds
	$$
	m_f(\Sigma^-_{t_2})\geq m_f(\Sigma^-_{t_1}).
	$$
	If equality holds, then $s^-_t=s^+_t$ for every $t\in [t_1,t_2)$ and, for every $t\in [t_1,t_2]$, there is $\Sigma_t\in \mathcal{E}_t$ such that  $\Sigma_t$ is a  union of flat disks and  $H(S)=0$ on $\partial \Sigma_t$.
\end{prop}
\begin{proof}
	Let $f\in C^\infty(\mathbb{R})$ be such that  $0< f'(0)< 1$ and
$
2\,	f''(0)\,f(0)\leq 1-f'(0)^2.
	$
	Note that 
	\begin{align} \label{comparison ineq}
		\frac{d}{ds}\bigg|_{s=0}\,\left[f(s)\,(1-f'(s)^2)\right]\geq 0.
	\end{align}
	
	In view of Lemma \ref{comparison function} and \eqref{comparison ineq}, using that $t_0>0$, the same reasoning as in \cite[Lemma 3]{bray} shows that the function
	$$
s\mapsto	\upsilon(s)\,(1-\bar \upsilon'(s)^2)
	$$
	is nondecreasing on $(|S(\Sigma^{+}_{t_0})|,\infty)$  in the sense of distributions. In particular, for all $s_1,\,s_2\in (|S(\Sigma^{+}_{t_0})|,\infty)$ with $s_1<s_2$, we have
	$$
		\upsilon(s_1)\,(1-\bar \upsilon'(s_1)^2)\leq 	\upsilon(s_2)\,(1-\bar \upsilon' (s_2)^2).
	$$
	Recall from Lemma \ref{rho k increasing} that $\upsilon$ is strictly increasing. Thus, we may choose a strictly decreasing  sequence $\{s_1^k\}_{k=1}^\infty$ and a strictly increasing sequence $\{s_2^k\}_{k=1}^\infty$ of numbers $s^k_1,\,s^k_2\in (|S(\Sigma^{-}_{t_1})|,|S(\Sigma^-_{t_2})|)$ with $s^k_1\searrow |S(\Sigma^-_{t_1})|$ and $s^k_2\nearrow  |S(\Sigma^-_{t_2})|$ such that $\upsilon'(s^k_1)$ and $\upsilon'(s^k_2)$ exist for all $k$. In view of \eqref{ld and rd}, we have
	$$
\frac{1}{\cosh(\sigma(s^k_1))^2}\,\upsilon(s^k_1)\leq \frac{1}{\cosh(\sigma(s^k_2))^2}\,\upsilon(s^k_2)
	$$
	for all $k$ with $s^k_1<s^k_2$. 	 Using that  $\sigma$ and $\upsilon$ are continuous, we conclude that 
	\begin{align*} 
\frac{1}{\cosh (t_1)^2}	\,|\Sigma^-_{t_1}|\leq\frac{1}{\cosh (t_2)^2}\,|\Sigma^-_{t_2}|. 
	\end{align*} 
	If equality holds, then,  using Lemma \ref{comparison function} and repeating the argument in the proof of \cite[Lemma 3]{bray}, we see that $s^-_t=s^+_t$ for every $t\in[t_1,t_2)$ and that, for almost every $t\in(t_1,t_2)$, every $\Sigma_t\in \mathcal{E}_t$ is a union of  flat disks and that $H(S)=0$ on $\partial \Sigma_t$ for every $\Sigma_t\in \mathcal{E}_t$. Using Lemma \ref{compactness lemma}, we conclude that, for every $t\in[t_1,t_2]$, there is $\Sigma_t\in \mathcal{E}_t$ such that  $\Sigma_t$ is a  union of flat disks and  $H(S)=0$ on $\partial \Sigma_t$. The assertion follows.
\end{proof}
\begin{rema}
Note that, if $t_0<0$ and $t\in(t_0,0)$, then, in view of \eqref{interpolate}, 	$m_f(\Sigma^-_t)>m_f(\Sigma^+_t)$ unless $s^-_t=s^+_t$. For this reason, we should not expect $m_f(\Sigma_t)$ to be nondecreasing in the case where $t_0<0$.  \end{rema}
\begin{coro} \label{monotonicity corollary}
	Suppose that  
	\begin{align} \label{outward minimizing} 
\text{$J_{t_0}(\Sigma_{t_0})\leq J_{t_0}(\Sigma)$ for every admissible surface $\Sigma \subset \bar M(S)$ with $\Omega(\Sigma_{t_0})\subset \Omega(\Sigma)$}
	\end{align} 
	with equality if and only if $\Sigma=\Sigma_{t_0}$.
Then, for all $t_1,\,t_2\in[t_0,\infty)$ with $t_2>t_1$, there holds
	\begin{align} \label{monotnicity to be checked} 
	m_f(\Sigma^-_{t_1})\leq m_f(\Sigma^-_{t_2}). 
	\end{align} 
	If equality holds, then $s^-_t=s^+_t$ for every $t\in [t_1,t_2)$ and, for every $t\in [t_1,t_2]$, there is $\Sigma_t\in \mathcal{E}_t$ such that  $\Sigma_t$ is a  union of flat disks and  $H(S)=0$ on $\partial \Sigma_t$.
\end{coro}
\begin{proof}
	In view of Proposition \ref{montonicity prop}, it remains to check \eqref{monotnicity to be checked} in the case where $t_1=t_0$. By Lemma \ref{sharp initial} and \eqref{outward minimizing}, we have $\mathcal{E}_{t_0}=\{\Sigma_{t_0}\}$. By Proposition \ref{montonicity prop}, for every strictly decreasing sequence $\{\tilde t_k\}_{k=1}^\infty$ with $\tilde t_k\to t_0$, 
	$$
	m_f(\Sigma^-_{\tilde t_k})\leq m_f(\Sigma^-_{t_2}).
	$$
	Using Lemma \ref{compactness lemma}, the assertion follows. 
\end{proof}
\section{Asymptotic behavior of the free energy mass}
In this section, we assume that $ S\subset \mathbb{R}^3$ is  a properly embedded plane with integrable mean curvature. 
Moreover, we assume that the complement of a compact subset of $S$ is contained in the graph of a function $\psi \in C^\infty(\mathbb{R}^2)$ such that, for some $\tau>1/2$,
\begin{equation} 
	\begin{aligned} \label{asymptotically flat section 6}
		\sum_{i=1}^2|\partial_i\psi(y)|+\sum_{i,\,j=1}^2|y|\,|\partial^2_{ij}\psi(y)|=O(|y|^{-\tau}).
	\end{aligned}
\end{equation} 
Let $\nu(S)$ be the unit normal field of $S$ asymptotic to $-e_3$. The component of $\mathbb{R}^3\setminus S$ that $\nu(S)$ points out of is denoted by $M(S)$ and referred to as the region above $S$.  We denote the second fundamental form of $S$ with respect to $\nu(S)$ by $h(S)$. Our convention is such that the trace of $h(S)$, i.e., the mean curvature $H(S)$, is the tangential divergence of $\nu(S)$.

Recall from Appendix \ref{appendix:af surfaces} the definition \eqref{exterior mass2} of the exterior mass $m$ of $S$.

Recall from Appendix \ref{appendix:af surfaces} the definition of an admissible surface $\Sigma \subset \bar M(S)$ as well as that of its lateral surface $S(\Sigma)\subset S$ and its inside $\Omega(\Sigma)\subset M(S)$. Moreover, recall the definition of the normal $\nu(\Sigma)$, the co-normal $\mu(\Sigma)$, and the co-normal $\mu(S(\Sigma))$, as well as  the conventions for the second fundamental form $h(\Sigma)$ and mean curvature $H(\Sigma)$.

Recall from Appendix \ref{appendix:free energy} the definition of a minimal capillary surface $\Sigma\subset \bar M(S)$ with capillary angle $\theta\in(0,\pi)$ and the notion of stability for such surfaces. Given $t\in\mathbb{R}$, recall the definition   \eqref{free energy smooth} of the free energy $J_t(\Sigma)$ with angle $\arccos(-\tanh (t))$ of an admissible surface $\Sigma \subset \bar M(S)$. Moreover, recall the definition \eqref{free energy mass} of the  free energy mass $m_f(\Sigma)$ of a minimal capillary surface $\Sigma\subset \bar M(S)$.

Let $t_0\geq 0$. Suppose that $\Sigma_{t_0}\subset \bar M(S)$ is a stable  minimal capillary surface with capillary angle $\arccos(-\tanh (t_0))$. We assume that $H(S)\geq 0$ on  of $S\setminus S(\Sigma_{t_0})$.

Given $t\in[t_0,\infty)$, recall from Section \ref{stable minimal section} the definition of the sets $\mathcal{E}_{t}$ of all minimal capillary surfaces  $\Sigma_t\in \mathcal{E}_t$ satisfying \eqref{mathE def}. Moreover, recall the definition of $\Sigma^-_t\in \mathcal{E}_t$, whose existence is shown in Proposition \ref{universal existence}. 

The goal of this section is to show that 
$$
\lim_{t\to\infty} m_f(\Sigma^-_t)=m.
$$

\begin{lem} \label{coarse area estimate}
	For every $t\geq t_0$ and $r\geq 1$, there holds
	$$
	|\{x\in \Sigma^-_t:|x|\leq r\}|  \leq 4\,\pi\,r^2.
	$$
\end{lem}
\begin{proof}
	Given $t\geq t_0$, assume first that  $r\geq 1$ is such that the intersection of  $\Sigma^-_t$ with $\{x\in \bar M(S):|x|= r\}$ and the intersection of  $S$ with $\{x\in \bar M(S):|x|= r\}$ are both transverse. Let $\tilde \Sigma_t\subset \bar M(S)$ be the admissible Lipschitz surface with $\Omega(\tilde \Sigma_t)=\Omega(\Sigma^-_t)\cup \{x\in \bar M(S):|x|< r\}.$
	Note that $\Omega(\Sigma_{t_0})\subset \Omega(\tilde \Sigma_t)$ and $S(\Sigma^-_t)\subset S(\tilde \Sigma_t)$. Moreover, there holds $$\{x\in \Sigma^-_t:|x|> r\}=\{x\in \tilde \Sigma_t:|x|> r\}$$ 
	and 
	$$
	|\{x\in \tilde \Sigma_t:|x|\leq r\}|\leq 4\,\pi\,{r}^2. 
	$$
	By \eqref{mathE def}, $J_t(\Sigma^-_t)\leq J_t(\tilde \Sigma_t)$. Since $\Phi(t)>0$ for every $t>0$, we obtain $|\Sigma^-_t|\leq |\tilde \Sigma_t|$ so that 
	$$	|\{x\in \Sigma^-_t:|x|\leq r\}|  \leq 4\,\pi\,r^2,$$ as asserted. For general $r\geq 1$, the assertion follows by approximation. 
\end{proof}

Given $t\geq t_0$, let $$\alpha_t=\inf\{|x|:x\in \Sigma^-_t\}\qquad \text{and}\qquad \beta_t=\sqrt{\frac{|\Sigma^-_t|}{\pi}}.$$ 
\begin{lem} \label{inner radius to infty}
There holds
	$
\alpha_t\to \infty
$
and $\beta_t\to\infty$ as $t\to\infty$. 

\end{lem}
\begin{proof} 
By Corollary \ref{inner divergence}, we have $\alpha_t\to\infty$  and $|S(\Sigma^-_t)|\to\infty$. By \eqref{mathE def}, we have 
$$
J_{t_0+1}(\Sigma^-_{t_0+1})\leq J_{t_0+1}(\Sigma^-_t)
$$
for every $t\in(t_0,\infty)$.
Using that $\Phi(t_0+1)>0$, we conclude that  $\beta_t\to\infty,$ as asserted. 
\end{proof} 

\begin{lem} \label{layer cake integral estimate}
	Let $\varepsilon>0$ and $X\in C^\infty(\mathbb{R}^3,\mathbb{R}^3)$ be such that
	$DX(x)=O(|x|^{-2-\varepsilon})$  as $x\to\infty$. As $t\to\infty$,  there holds
	$$
	\int_{\partial \Sigma^-_t}\langle X,\mu(\Sigma^-_t)\rangle =o(1).
	$$
\end{lem}
\begin{proof}

By the first variation formula, using that $H(\Sigma^-_t)=0$, 
	$$
	\int_{\partial\Sigma^-_t}\langle X,\mu(\Sigma^-_t)\rangle=\int_{\Sigma^-_t}\operatorname{div }X-\int_{\Sigma^-_t}\langle D_{\nu(\Sigma^-_t)}X,\nu(\Sigma^-_t)\rangle= O(1)\,\int_{\Sigma^-_t}|x|^{-2-\varepsilon}.
	$$
In the case where $\beta_t\leq \alpha_t$, we have  
$$
\int_{\Sigma^-_t}|x|^{-2-\varepsilon}\leq \beta_t^{-2-\varepsilon}\,|\Sigma^-_t|= \pi\,\beta_t^{-\varepsilon}. 
$$
Assume that $\beta_t\geq \alpha_t$. Note that 
\begin{align*} 
&\int_{\Sigma^-_t}|x|^{-2-\varepsilon}
 =(2+\varepsilon)\,\int_{\alpha_t}^{\infty}s^{-3-\varepsilon}\,|\{x\in \Sigma^-_t:|x|\leq s\}|.
\end{align*} 
	Using Lemma \ref{coarse area estimate}, we have
	$$
	\int_{\alpha_t}^{\beta_t}s^{-3-\varepsilon}\,|\{x\in \Sigma^-_t:|x|\leq s\}|\leq 4\,\pi\, \varepsilon^{-1}\,(\alpha_t^{-\varepsilon}-\beta_t^{-\varepsilon}).
	$$
Moreover, 
	$$
	\int_{\beta_t}^{\infty}s^{-3-\varepsilon}\,|\{x\in \Sigma^-_t: |x|\leq s\}|\leq \pi\, (2+\varepsilon)^{-1}\,\beta_t^{-\varepsilon}.
	$$
The assertion follows from these identities and Lemma \ref{inner radius to infty}.
\end{proof}
Given $t\geq t_0$ sufficiently large, let $$P_t=\{(y,0):y\in\mathbb{R}^2\text{ and }(y,\psi(y))\in (S\setminus S(\Sigma^-_t))\cup \partial \Sigma^-_t\}.$$  By \eqref{asymptotically flat section 6}  and Lemma \ref{inner radius to infty}, the horizontal projection $$p_t:(S\setminus S(\Sigma^-_t))\cup \partial \Sigma^-_t \to P_t$$  given by $p_t(x)=x-\langle x,e_3\rangle\, e_3$ is a diffeomorphism. We tacitly identify functions and maps defined on $(S\setminus S(\Sigma^-_t))\cup \partial \Sigma^-_t$ with functions and maps defined on $P_t$ by precomposition with $p_t$. Let $\mu(P_t)$ be the inward co-normal of $\partial P_t\subset P_t$.

\begin{prop} \label{asymptotics prop}
	There holds
	$$
	\lim_{t\to\infty} m_f(\Sigma^-_t)=m.
	$$
\end{prop}
\begin{proof}
	For notational convenience,  we extend $\psi$ to a smooth function on $\mathbb{R}^3$ that is independent of the third variable.
	
	By the optimal isoperimetric inequality for minimal disks \cite[p.~1]{Carleman}, using Lemma \ref{capillary topology},
	$$
	m_f(\Sigma^-_t)=\frac{1}{\cosh (t)}\,\sqrt{\frac{|\Sigma^-_t|}{\pi}}\leq\frac{1}{2\,\pi\,\cosh (t)}\,|\partial \Sigma^-_t|. 
	$$ 
	By Lemma \ref{angles}, we have
	\begin{align} 
\frac{1}{\cosh (t)}=\langle\nu(S),\mu( \Sigma^-_t)\rangle.
	\end{align}  
 There holds
	$$
	\nu(S)=\frac{1}{\sqrt{1+|D \psi|^2}}\,\left(D\psi-e_3\right )
	$$
	 on $S\setminus S(\Sigma^-_t)\cup \partial \Sigma^-_t $ 
	provided that $t>t_0$ is sufficiently large. Using \eqref{asymptotically flat section 6}, we see that there is $\eta\in C^\infty(\mathbb{R}^3)$ with $D\eta=O(|x|^{-1-2\,\tau})$ such that $$\eta(y,0)=\frac{1}{\sqrt{1+|D\psi(y,0)|^2}}$$ for every $y\in\mathbb{R}^2$. 
	Using  Lemma \ref{layer cake integral estimate}, we obtain that 
	$$
	\int_{\partial \Sigma^-_t}\frac{1}{\sqrt{1+|D \psi|^2}}\,\langle e_3,\mu( \Sigma^-_t)\rangle=\int_{\partial \Sigma^-_t}\eta\,\langle e_3,\mu( \Sigma^-_t)\rangle =o(1).
	$$
	By Lemma \ref{angles}, we have that 
	$$
	\mu( \Sigma^-_t)=\tanh (t)\,\mu( S(\Sigma^-_t))+\frac{1}{\cosh (t)}\,\nu(S)
	$$
 on $\partial \Sigma^-_t$.	Note that, by \eqref{asymptotically flat section 6} and Lemma \ref{inner radius to infty},
	$$
	\int_{\partial \Sigma^-_t}\frac{1}{\sqrt{1+|D \psi|^2}}\,\langle D\psi ,\nu( S)\rangle%=\int_{\partial \Sigma^-_t}\frac{|D \psi|^2}{1+|D \psi|^2}
	=o(1)\,|\partial \Sigma^-_t|.
	$$
We conclude that
$$
\frac{1}{\cosh (t)}\,|\partial \Sigma^-_t|=(1+o(1))\, \int_{\partial \Sigma^-_t}\frac{1}{\sqrt{1+|D \psi|^2}}\,\langle D\psi,\mu( S(\Sigma^-_t))\rangle +o(1)
$$
 as $t\to\infty$.
Note that 
$$
\int_{\partial \Sigma^-_t}\frac{1}{\sqrt{1+|D \psi|^2}}\,\langle D \psi,\mu( S(\Sigma^-_t))\rangle=\int_{\partial P_t}\frac{1}{1+| D \psi|^2}\,\langle D \psi,\mu(P_t)\rangle. 
$$
Recall that, on $S\setminus S(\Sigma^-_t) $, 
$$
H(S)=\sum_{i=1}^2\partial_i \left(\frac{\partial_i \psi}{\sqrt{1+|D\psi|^2}}\right).
$$
Using that $H(S)$ is integrable as well as  \eqref{asymptotically flat section 6} and Lemma \ref{inner radius to infty}, we conclude that 
$$
\int_{ P_t}(\partial_1\partial_1 \psi+\partial_2\partial_2  \psi)=o(1).
$$
In conjunction with the  first variational formula, using  \eqref{asymptotically flat section 6} and Lemma \ref{inner radius to infty} again, we see that
\begin{align*} 
\int_{\partial P_t}\frac{1}{1+| D \psi|^2}\,\langle D \psi,\mu(P_t)\rangle
=\lim_{r\to\infty} \frac{1}{r}\,\int_{\left\{y\in\mathbb{R}^2:|y|	 =r\right\}}\sum_{i=1}^2 y_i\,\partial_i\psi +o(1).
\end{align*} 
In view of \eqref{exterior mass2}, the assertion follows.
\end{proof}

\section{Proofs of Theorem \ref{extrinsic:rpi} and Theorem \ref{extrinsic:pmt}}

Recall from Appendix \ref{appendix:af surfaces} the definition of an asymptotically flat support surface $S\subset \mathbb{R}^3$ with unit normal $\nu(S)$, the definition of the region $M(S)$ above $S$,   the conventions for the second fundamental form $h(S)$ and the mean curvature $H(S)$, and the definition of the  Gauss curvature $K(S)$. Moreover, recall the definitions of the exterior mass $m$, the outermost free boundary minimal surface $D\subset \bar M(S)$ supported on $S$, the exterior surface $S'\subset S$ of $S$, and the exterior region $M'(S)\subset M(S)$ of $S$. 

 Recall from Appendix \ref{appendix:af surfaces} the definition of an admissible surface $\Sigma \subset \bar M(S)$, of its lateral  surface $S(\Sigma)\subset S$, and of its inside $\Omega(\Sigma)\subset M(S)$. 
 
 Recall from Appendix \ref{appendix:free energy} the definition of a minimal capillary surface $\Sigma\subset \bar M(S)$ with capillary angle $\theta\in (0,\pi)$ and the notion of stability for such a surface. Moreover, recall the definition of the free energy mass $m_f(\Sigma)$ of a minimal capillary surface $\Sigma\subset \bar M(S)$.

\begin{proof}[Proof of Theorem \ref{extrinsic:rpi}] 
Let $S\subset \mathbb{R}^3$ be an asymptotically flat support surface with  outermost free boundary minimal surface $D\subset \bar M(S)$, exterior surface $S'\subset S$, and exterior region $M'(S)\subset M(S)$. Suppose that $H(S)\geq 0$ on $S'$. By Lemma \ref{topology}, we may assume that $S$ is diffeomorphic to the plane.    
Given $t\in[t_0,\infty)$, recall from Section \ref{stable minimal section} the definition of the sets $\mathcal{E}_{t}$ of all minimal capillary surfaces  $\Sigma_t\in \mathcal{E}_t$ satisfying \eqref{mathE def}. Recall from \eqref{s def} the definitions of $s^-_t,\,s^+_t$ and that, according to Proposition \ref{universal existence}, there are $\Sigma^-_t,\,\Sigma^+_t\in\mathcal{E}_t$ with $|S(\Sigma^-_t)|=s^-_t$ and $|S(\Sigma^+_t)|=s^+_t$. Moreover, recall from Appendix \ref{appendix:af surfaces} the definition of the unit tangent field $e(\partial \Sigma_t)$ of $\Sigma_t$. 
By Corollary \ref{monotonicity corollary} and Lemma \ref{minimization in homology class}, we have, for every $t>0$,
$
m_f(\Sigma^-_t)\geq m_f(D).
$
In conjunction with  Proposition \ref{asymptotics prop}, we conclude that, as asserted,
$$
m\geq m_f(D)=\sqrt{\frac{|D|}{\pi}}.
$$
Assume now that equality holds. By Corollary, \ref{monotonicity corollary} $s^-_t=s^+_t$ for every $t\geq 0$. Moreover, for every $t\geq 0$ there is $\Sigma_t\in\mathcal{E}_t$ such that $\Sigma_t$ is a union ot flat disks and $H(S)=0$ on $\partial \Sigma_t$. Suppose, for a contradiction, that there are $t_1,\,t_2\geq 0$ with $t_2>t_1$ such that $\Sigma_{t_1}\cap \Sigma_{t_2}\neq \emptyset$. Since $\Sigma_{t_1}$ and $\Sigma_{t_1}$ are both unions of flat disks, it follows that $\partial \Sigma_{t_1}\cap \partial \Sigma_{t_2}\neq \emptyset$ and that the intersection of $\Sigma_{t_1}$ and $\Sigma_{t_2}$ is transverse. Moreover, since $\Sigma_{t_1}$ and $\Sigma_{t_2}$ are both unions of flat disks that intersect $S$ at a constant angle, it follows that both $e(\partial \Sigma_{t_1})$ and $e(\partial \Sigma_{t_2})$ are principal directions of $S$ with positive principal curvatures, respectively. Consequently, $h(S)$ is positive definite on  $\partial \Sigma^1_{t_1}\cap \partial \Sigma^1_{t_2}$.  This is a contradiction, since $H(S)=0$ on $\partial \Sigma_{t_1}$. It follows that  $\Sigma_{t_1}\cap \Sigma_{t_2}= \emptyset$ for all $t_1,\,t_2\geq 0$ with $t_2>t_1$. In conjunction with  Lemma \ref{compactness lemma}, we conclude that,  for every $t\geq 0$, $\Sigma_{t'}\to \Sigma_t$ smoothly as $t'\to t$.  It follows that the number of  components of $\Sigma_t$ is constant. Since $S$ is diffeomorphic to the plane and $\Sigma_t$ is diffeomorphic to a union of disks, the number of  components of $\bar \Omega(\Sigma_t)$ equals the number of components of $\Sigma_t$. Let $K\subset \bar M(S)$ be a connected compact set with $\bar \Omega(D)\subset K$. By the intermediate value theorem, for every $t\geq 0$, every component of $\bar \Omega(\Sigma_t)$ intersects $K$. Using Lemma \ref{inner radius to infty}, we see that $K\subset \bar \Omega (\Sigma_t)$ provided that $t>0$ is sufficiently large. It follows that both $\bar \Omega(\Sigma_t)$ and $\Sigma_t$ are connected for every $t\geq 0$. By the intermediate value theorem and Lemma \ref{inner divergence}, using that $\Sigma_0=D$, we have
$$
S'=\bigcup_{t\in[0,\infty)}\partial \Sigma_t.
$$
Consequently, $S'$ is a minimal surface with boundary $\partial D$ that is diffeomorphic to $\{y\in\mathbb{R}^2:1\leq |y|<2\}$. % 
 Moreover, since $D$ is a flat disk that intersects $S$ orthogonally, the geodesic curvature of $\partial D\subset S$ vanishes. By the Gauss-Bonnet theorem, using also \eqref{asymptotically flat section 2},
$$
\int_{S'}K(S)=-2\,\pi.
$$   
Reflecting $S'$ across the plane that contains $D$, we obtain a connected complete  minimal surface $\hat S$ immersed in $ \mathbb{R}^3$ that is diffeomorphic to $\{y\in\mathbb{R}^2:1< |y|<2\}$ and satisfies 
$$
\int_{\hat S}K(\hat S)=-4\,\pi.
$$
By a result of R.~Osserman \cite[p.~341]{Osserman},  $\hat S$ is a catenoid. It follows that $S'$ is a half-catenoid, as asserted. 
\end{proof}

\label{section:proof of main results 2}

\begin{proof}[Proof of Theorem \ref{extrinsic:pmt}] Let $S\subset \mathbb{R}^3$ be an asymptotically flat support surface.  Suppose that $H(S)\geq 0$ on $S$. If there is a compact minimal surface $\Sigma\subset \bar M(S)$ that intersects $S$ orthogonally along $\partial \Sigma$, then, in view of Lemma \ref{existence of outermost free boundary} and Theorem \ref{extrinsic:rpi}, we have $m>0$, as asserted.
	Using Lemma \ref{existence of outermost free boundary}, we may therefore assume that $S$ is diffeomorphic to the plane. Moreover, we may assume that $S$ is not a flat plane. By the Gauss-Bonnet theorem, using \eqref{asymptotically flat section 2}, we see that
	$$
	\int_{S}K(S)=0. 
	$$
Since $S$ is not a flat plane, it follows that there is $x\in S$  with $K(S)(x)>0$. Let $\varepsilon>0$ be such that the plane $$P_{\varepsilon}=\{x'\in \mathbb{R}^3:\langle x'-x+\varepsilon\,\nu(S)(x),\nu(S)(x)\rangle=0\}$$ and $S$ intersect transversely and let $\Sigma_{\varepsilon}$ be the  component of $\bar M(S)\cap P_{\varepsilon}$ with $x\in S(\Sigma_{\varepsilon})$.  Using that $K(S)(x)>0$, we see that
\begin{align} \label{negative free energy} 
J_{t_0}(\Sigma_{\varepsilon})<0
\end{align} 
provided that $t_0>0$ is sufficiently large and that $\varepsilon>0$ is sufficiently small. Arguing as in the proof of Proposition \ref{capillary existence} and using \eqref{negative free energy}, we conclude that there is a stable minimal capillary surface $\Sigma_{t_0}\subset \bar M(S)$ with capillary angle $\arccos(-\tanh(t_0))$ such that, for every admissible surface $\Sigma\subset \bar M(S)$,
	$
	J_{t_0}(\Sigma_{t_0})\leq J(\Sigma). 
	$ 
	 Clearly
	$
	m_f(\Sigma_{t_0})>0.
	$   
Arguing as in the proof of Theorem \ref{extrinsic:rpi}, we see that
$
m\geq m_f(\Sigma_{t_0})>0,
$
as asserted.
\end{proof}

\section{Embedded minimal surfaces  with finite total curvature} 
In this section, we assume that $\tilde S$ is a connected complete minimal surface embedded in $\mathbb{R}^3$ with designated unit normal $\nu(\tilde S)$. Let $K(\tilde S)$ be the Gauss curvature of $\tilde S$. We assume that 
$$
\int_{\tilde S} |K(\tilde S)|<\infty. 
$$   
The goal of this section is to prove Theorem \ref{catenoid characterization}.

Recall from \cite[Proposition 1]{Schoen} that, rotating $\tilde S$ if necessary,  there is $\lambda>1$ and an integer $\ell\geq 1$ such that $\{x\in \mathbb{R}^3:|x|=\lambda\}$ and $\tilde S$ intersect transversely in Jordan curves $\Gamma_1,\ldots, \Gamma_\ell$ and that the components $\tilde S_1,\ldots,\tilde S_\ell$ of   $\{x\in \tilde S:|x|>\lambda\}$ bounded by $\Gamma_1,\ldots \Gamma_\ell,$ respectively, are contained in the graphs of functions $\psi_1,\ldots,\psi_\ell \in C^\infty(\mathbb{R}^2)$ with $\psi_1(y)<\ldots< \psi_\ell(y)$ for every $y\in \mathbb{R}^2$. Moreover, for each $1\leq k\leq \ell$, 
\begin{equation}  
	\begin{aligned} \label{asymptotically flat section 8}
		|\psi_k(y)-a_k\,\log|y|-b_k|+	\sum_{i=1}^2|\partial_i\psi_k(y)|+\sum_{i,\,j=1}^2|y|\,|\partial^2_{ij}\psi_k(y)|=O(|y|^{-1})
	\end{aligned}
\end{equation} 
for some $a_k,\,b_k\in \mathbb{R}.$ 
Note that $\mathbb{R}^3\setminus \tilde S$ has exactly two components. We denote the component whose outward normal points in direction of $-e_3$ on $\tilde S_\ell$ by $M^+$ and the component whose outward normal points in direction of $e_3$ on $\tilde S_1$ by $M^-$. Note that $M^+$ and $M^-$  coincide if and only if  $\ell$ is even. 

Given $1\leq k\leq \ell$, let $\mu(\tilde S_k)$ be the inward co-normal of $\tilde S_k$. Following \cite[\S17]{Fang}, we define the flux $\operatorname{Flux}(\tilde S_k)\in\mathbb{R}^3$ of $\tilde S_k$ by 
\begin{align} \label{Flux} 
	\operatorname{Flux}(\tilde S_k)=\int_{\Gamma_k}\mu(\tilde S_k).
\end{align} 
Recall that, by the first variation formula, for every Jordan curve $\Gamma\subset \tilde S_k$ that is homologous to $\Gamma_k$ in $\tilde S_k$,
\begin{align} \label{homotopy invariant} 
	\operatorname{Flux}(\tilde S_k)=\int_{\Gamma} \mu(\tilde S_k).
\end{align}

Let $\Sigma^+$ be the  component of $\{x\in \bar M^+:|x|=\lambda\}$ that intersects $\tilde S_\ell$ transversely  and $\Sigma^-$ be the  component of $\{x\in \bar M^-:|x|=\lambda\}$ that intersects $\tilde S_1$ transversely. Note that both $M^+\setminus \Sigma^+$ and $M^-\setminus \Sigma^-$ have exactly two components. Let $\gamma^+(\tilde S)$ be the infimum of $\sqrt{4\,\pi\,|\Sigma|}$ among all compact surfaces $\Sigma\subset \bar M^+$ that intersect $\tilde S$ transversely and are homologous to $\Sigma^+$ relative to $\tilde S$. Likewise, let $\gamma^-(\tilde S)$ be the infimum of $\sqrt{4\,\pi\,|\Sigma|}$ among all compact surfaces $\Sigma\subset \bar M^-(S)$ that intersect $\tilde S$ transversely and are homologous to $\Sigma^-$ relative to $\tilde S$. We define the neck size $\gamma(\tilde S)$ of $\tilde S$ as 
\begin{align}  \label{neck size}
	\gamma(\tilde S)=\max\{\gamma^+(\tilde S),\gamma^-(\tilde S)\}.
\end{align} 

For the proof of Theorem \ref{catenoid characterization}, recall from Appendix \ref{appendix:af surfaces} the definition of an asymptotically flat support surface $S\subset \mathbb{R}^3$ with unit normal $\nu(S)$ and the definition of the region $M(S)$ above $S$. Moreover, recall the definitions of the exterior mass $m$, the outermost free boundary minimal surface $D\subset \bar M(S)$ supported on $S$, the exterior surface $S'\subset S$ of $S$, and  the exterior region $M'(S)\subset M(S)$ of $S$. 

\begin{proof}[Proof of Theorem \ref{catenoid characterization}] We may assume that $\gamma^+(\tilde S)\geq \gamma^-(\tilde S)$. In view of \eqref{asymptotically flat section 8}, outside a compact set, $\bar M^+$ is foliated by compact surfaces with positive mean curvature that meet $S$ at an acute angle. Using standard tools from geometric measure theory, it follows that either $\Sigma^+$ is null homologous relative to $\tilde S$ or that there is a nonempty compact free boundary minimal surface $\tilde D\subset \bar M^+$ with $\gamma^+(\tilde S)=\sqrt{4\,\pi\,|\tilde D|}$. If $ \Sigma^+$ is null homologous relative to $\tilde S$, then $\ell=1$ and, by the strong maximum principle, using \eqref{asymptotically flat section 8}, $\tilde S$ is a flat plane so that $\operatorname{Flux}(\tilde S)=(0,\,0,\,0)$ and $\gamma(\tilde S)=0$.

	If $\Sigma^+$ is not null homologous relative to $\tilde S$, let $\tilde S'$ be the  component of $\tilde S\setminus \partial \tilde D$ with $\tilde S'\cap \tilde S_\ell\neq \emptyset$. We extend $\tilde S'$ to an  asymptotically flat support surface $S\subset \mathbb{R}^3$ such that $\tilde D\setminus \partial\tilde  D\subset M(S)$. Since $\tilde D$ minimizes area in its homology class relative to $\tilde S$, we see that the outermost free boundary minimal surface $D\subset \bar M(S)$ of $S$ exists and satisfies $|D|\geq |\tilde D|$. Moreover, the exterior surface $S'$ of $S$ satisfies $S'\subset \tilde S'$. Moreover,  using \eqref{asymptotically flat section 8}, \eqref{homotopy invariant}, and \eqref{exterior mass2}, we see that 
	$\operatorname{Flux}(F_\ell)=(0,0,2\,\pi\, m)$ where $m$ is the exterior mass of $S$. By Theorem \ref{extrinsic:rpi}, we have
	\begin{align} \label{catenoid char eq} 
		m\geq\sqrt{\frac{|D|}{\pi}}
	\end{align} 
	so that $|\operatorname{Flux}(\tilde S_\ell)|\geq\gamma(\tilde S)$ as asserted.  If $|\operatorname{Flux}(\tilde S_\ell)|=\gamma(\tilde S)$, then  equality holds in \eqref{catenoid char eq}. By Theorem \ref{extrinsic:rpi}, $S'\subset \tilde S$ is a half-catenoid. By unique continuation, $\tilde S$ is a catenoid; see \cite[p.~235]{aronszajn}. The assertion follows. 
\end{proof}

\begin{appendices} 
	\section{Asymptotically flat support surfaces}
	\label{appendix:af surfaces}
	In this section, we recall background on asymptotically flat support surfaces.
	
	Let $ S\subset \mathbb{R}^3$ be a connected  unbounded properly embedded surface such that $\mathbb{R}^3\setminus S$ has exactly two components.  In particular, $S$ has no boundary. We say that $S$ is an asymptotically flat support surface if its mean curvature is integrable and the complement of a compact subset of $S$ is contained in the graph of a function $\psi \in C^\infty(\mathbb{R}^2)$ such that, 	for some $\tau>1/2$,
	\begin{equation} \label{asymptotically flat section 2}
		\begin{aligned} 
			\sum_{i=1}^2|\partial_i\psi(y)|+\sum_{i,\,j=1}^2|y|\,|\partial^2_{ij}\psi(y)|=O(|y|^{-\tau}).
		\end{aligned}
	\end{equation} 
  Let $\nu(S)$ be the unit normal field asymptotic to $-e_3$. The component of $\mathbb{R}^3\setminus S$ that $\nu(S)$ points out of is denoted by $M(S)$ and referred to as the region above $S$.  We denote the second fundamental form of $S$ with respect to $\nu(S)$ by $h(S)$. Our convention is such that the trace of $h(S)$, i.e., the mean curvature $H(S)$, is the tangential divergence of $\nu(S)$.

	  The exterior mass of $S$ is
	\begin{align} \label{exterior mass2}
		m=\lim_{r\to\infty}\frac{1}{2\,\pi\,r}\,\int_{\{y\in\mathbb{R}^2:|y|=r\}}\sum_{i=1}^2\,y^i\,\partial_i\psi(y).
	\end{align}

	A nonempty compact minimal surface  $D\subset \bar M(S)$  that intersects $S$ orthogonally along $\partial D$ such that $\partial D\subset S$ bounds a bounded open subset of $S$ is called a free boundary minimal surface. Given a free boundary minimal surface $D\subset \bar M(S)$, let $S'\subset S$ be the closure of the unbounded component of $S\setminus \partial D$ and $M'(S)$ be the unbounded component of   $M(S)\setminus D$. We call $D$ outermost if every compact minimal surface  $\Sigma\subset \bar M'(S)$  that intersects $S$ orthogonally along $\partial \Sigma$ is contained in $D$. If an outermost free boundary minimal surface $D\subset \bar M(S)$ exists, we say that $S$ is an asymptotically flat support surface with  outermost free boundary surface $D$, exterior surface $S'$, and exterior region $M'(S)$.

	\begin{lem} \label{existence of outermost free boundary}
		Let $S$ be an asymptotically flat support surface. Suppose that $H(S)\geq 0$. Then either there is a unique outermost free boundary minimal surface $D\subset \bar M(S)$ supported on $S$ or there is no compact minimal surface $\Sigma\subset \bar M(S)$ that intersects $S$ orthogonally along $\partial \Sigma$ and $S$ is diffeomorphic to the plane.
	\end{lem}
	\begin{proof}
		This follows from \cite[Lemma 2.1]{Koerber} and \cite[Lemma 2.3]{Koerber}.
	\end{proof}
	\begin{lem} \label{topology}  Let 	$S\subset\mathbb{R}^3$ be an asymptotically flat support surface with outermost free boundary minimal surface $D\subset \bar M(S)$ and exterior surface $S'\subset S$. Suppose that $H(S)\geq 0$ on $S'$. Then $S'$ is diffeomorphic to the plane minus a finite number of  disks and  $D$ is diffeomorphic to a union of disks.
	\end{lem}
	\begin{proof}
		This follows from \cite[Lemma 2.1]{Koerber} and \cite[Lemma 2.3]{Koerber}.
	\end{proof}
	
	In view of Lemma \ref{existence of outermost free boundary} and Lemma \ref{topology}, we may and will assume that $S$ is diffeomorphic to the plane.
	
	We say that a compact surface $\Sigma\subset \bar M(S)$ is admissible if $\Sigma$ intersects $S$ transversely along $\partial \Sigma$ and if $\Sigma$ has no closed components contained in $M(S)$. Given an admissible surface $\Sigma\subset \bar M(S)$, let $S(\Sigma)$ denote the closure of the union of the bounded components of $S\setminus \partial \Sigma$ and  $\Omega(\Sigma)\subset M(S)$ be the union of the bounded components of $M(S)\setminus \Sigma$. We call $S(\Sigma)$ the lateral surface of $\Sigma$ and $\Omega(\Sigma)$ the inside of $\Sigma$. Moreover, we call $|S(\Sigma)|$ the enclosed area of $\Sigma$. Note that every component of an admissible surface is an admissible surface. 
	
	Let $\nu(\Sigma)$ be the normal of $\Sigma$ pointing out of $\Omega(\Sigma),$  $\mu(\Sigma)$ be  the outward co-normal of $\partial \Sigma\subset \Sigma,$ $\mu(S(\Sigma))$ be  the outward co-normal of $\partial S(\Sigma)\subset S(\Sigma)$, and $e(\partial \Sigma)$ be the unit tangent field of $\partial \Sigma$ given by $e(\partial \Sigma)=\nu(\Sigma)\times \mu(\Sigma)$. Let $k(\partial \Sigma)$ be the geodesic curvature of $\partial \Sigma\subset \Sigma$ computed as the tangential divergence of  $\mu(\Sigma)$. We denote the second fundamental form and its trace, the mean curvature, of $\Sigma$ with respect to $\nu(\Sigma)$ by $h(\Sigma)$ and $H(\Sigma)$, respectively. Our convention is such that $H(\Sigma)$ is the tangential divergence of $\nu(\Sigma)$. Let $K(\Sigma)$ be the Gauss curvature of $\Sigma$. Let $\chi(\Sigma)$ be the Euler characteristic of $\Sigma$. Given a function $f\in C^\infty(\Sigma)$, we denote the gradient of $f$ by $\nabla^\Sigma f$ and the nonnegative Laplace-Beltrami operator of $f$ by $\Delta^\Sigma f$. 
	
		\begin{lem} \label{minimization in homology class}
	 Let 	$S\subset\mathbb{R}^3$ be an asymptotically flat support surface with outermost minimal free boundary  surface $D\subset \bar M(S)$, exterior surface $S'\subset S$, and exterior region $M'(S)$. Suppose that $H(S)\geq 0$ on $S'$.  If  $\Sigma \subset \bar M'(S)$ is an admissible surface with $\Omega(D)\subset \Omega(\Sigma)$, then  $|D|\leq | \Sigma|$ with equality if and only if $D=\Sigma$. 
	\end{lem}
	\begin{proof}
		This follows from \cite[Lemma 2.1]{Koerber} and \cite[Lemma 2.3]{Koerber}.
	\end{proof}
	
	\section{First and second variation of area and enclosed area}
	\label{appendix:variation}
In this section, we recall the formulas for the first and second variation of area and of  enclosed area of an admissible surface.
	
			Recall from Appendix \ref{appendix:af surfaces} the definition of an asymptotically flat support surface $S\subset \mathbb{R}^3$ with unit normal $\nu(S)$, the definition of the region $M(S)$ above $S$, and  the conventions for the second fundamental form $h(S)$ and the mean curvature $H(S)$.
	
	Let $S\subset \mathbb{R}^3$ be an asymptotically flat support surface diffeomorphic to the plane. Recall from Appendix \ref{appendix:af surfaces} the definition of an admissible surface $\Sigma \subset \bar M(S)$, of its lateral  surface $S(\Sigma)\subset S$, and of its inside $\Omega(\Sigma)\subset M(S)$. Moreover, recall the definitions of the normal $\nu(\Sigma)$, the co-normal $\mu( \Sigma),$ the co-normal $\mu(S(\Sigma))$, and the tangent field $e(\partial \Sigma)$, the conventions for the geodesic curvature $k(\partial \Sigma)$, second fundamental form $h(\Sigma),$ and the mean curvature $H(\Sigma)$, as well as the definition of the Gauss curvature $K(\Sigma)$. Given $f\in C^\infty(\Sigma)$, recall the definitions of $\nabla^\Sigma f$ and $\Delta^\Sigma f$.

	Let $\Sigma\subset \bar M(S)$ be an admissible surface and   $\theta \in C^\infty(\partial \Sigma,(0,\pi))$ be such that 
	$
	\cos(\theta)=\langle \nu(S),\nu(\Sigma)\rangle.
	$
	
	We first recall the following elementary results.
	\begin{lem}  \label{angles}
		There holds 
		\begin{equation*}
			\begin{aligned} 
				\nu(S)&=\cos(\theta)\,  \nu(\Sigma)+\sin(\theta) \,\mu(\Sigma),\qquad\qquad\qquad\qquad\qquad\qquad \\
				\mu(S(\Sigma))&=\sin(\theta) \,\nu(\Sigma)-\cos(\theta) \,\mu(\Sigma),\\
				\nu(\Sigma)&=\cos(\theta)\,  \nu(S)+\sin(\theta) \,\mu(S(\Sigma)),\text{ and}\\
				\mu(\Sigma)&=\sin(\theta) \,\nu(S)-\cos(\theta) \,\mu(S(\Sigma)).
			\end{aligned}
		\end{equation*}
		Moreover, we have
		\begin{align}  \label{geodesic} 
			k(\partial \Sigma)=\frac{1}{\sin(\theta)}\,h(S)(e(\partial \Sigma),e(\partial \Sigma))-\frac{\cos(\theta)}{\sin(\theta) }\,h(\Sigma)(e(\partial \Sigma),e(\partial \Sigma)).
		\end{align} 
	%	\begin{align}
	%	H(\partial \Sigma)=-\frac{1}{\sin(\theta)}\,\operatorname{tr}_{\partial \Sigma} h(\Sigma)\,\mu(S(\Sigma))-\frac{1}{\sin(\theta)}\,\operatorname{tr}_{\partial \Sigma} h(S)\,\mu(\Sigma).
	%	\end{align} 
	\end{lem}

	Let $\varepsilon>0$ and $F\in C^\infty(\Sigma\times(-\varepsilon,\varepsilon))\to \bar M(S)$ be such that $F(\,\cdot\,,0)=\operatorname{Id}_\Sigma$ and $F(x,s)\in S$ for every $x\in \partial \Sigma$ and $s\in(-\varepsilon,\varepsilon)$.  	Decreasing $\varepsilon>0$, if necessary, we obtain a smooth family $\{\Sigma(s):s\in(-\varepsilon,\varepsilon)\}$ of admissible surfaces $\Sigma(s)=\{F(x,s):x\in \Sigma\}$, which we call an admissible variation of $\Sigma$. We tacitly identify functions and maps defined on $\Sigma(s)$ with functions and maps defined on $\Sigma$ by precomposition with $F(\,\cdot\,,s)$. 
	 
	Let $f\in C^\infty(\Sigma)$ and $X\in C^\infty(\Sigma,\mathbb{R}^3)$ be such that $\langle X,\nu(\Sigma)\rangle =0$ and 
	$$
	\dot{F}=\frac{dF(\,\cdot\,,s)}{ds}\bigg|_{s=0}=f\,\nu(\Sigma)+X.
	$$ 
Reparametrizing $F$, we may assume that  $X|_{\partial \Sigma}$ is parallel to $\mu(\Sigma)$ so that 
  \begin{align*} 
  	 X=-\frac{\cos(\theta)}{\sin(\theta)}\,f\,\mu( \Sigma).
  	 \end{align*} 
  \begin{lem} \label{normal change}
  	Let $Z\in\mathbb{R}^3$. On $\partial \Sigma$, there holds
  	\begin{align*} 
  	\frac{d}{ds}\bigg|_{s=0}\langle \nu(\Sigma(s)), Z\rangle &=-\frac{\cos(\theta)}{\sin(\theta)}\,h(\Sigma)(\mu(\Sigma),\operatorname{proj}_{T\Sigma}Z)\,f-\langle \nabla^\Sigma f,\operatorname{proj}_{T\Sigma} Z\rangle \text{ and}\\
  	\frac{d}{ds}\bigg|_{s=0}\langle \nu(S), Z\rangle &=\frac{1}{\sin(\theta)}\,h(S)(\mu(S(\Sigma)),\operatorname{proj}_{TS}Z)\,f.
  	\end{align*} 
  \end{lem}
  \begin{proof}
  	Note that, on $\partial \Sigma$,
  	$$
  	\frac{d}{ds}\bigg|_{s=0}\langle \nu(\Sigma(s)), Z\rangle =h(\Sigma)(\operatorname{proj}_{T\Sigma}\dot F,\operatorname{proj}_{T\Sigma}Z)-\langle \nabla^\Sigma f,\operatorname{proj}_{T\Sigma} Z\rangle 
  	$$
  	and 
  	$$
  	\frac{d}{ds}\bigg|_{s=0}\langle \nu(S), Z\rangle =h(S)(\operatorname{proj}_{T\Sigma}\dot F,\operatorname{proj}_{TS} Z\rangle .
  	$$
  	Using Lemma \ref{angles}, we have, on $\partial \Sigma$,
  	\begin{align*} 
 \operatorname{proj}_{T\Sigma} \dot F=-\frac{\cos(\theta)}{\sin(\theta)}\,f\,\mu(\Sigma)\qquad\qquad \text{ and } \qquad\qquad   \dot F=\frac{1}{\sin(\theta)}\,f\,\mu(S(\Sigma)).
  	\end{align*} 
  	The assertion follows from these identities.
  \end{proof}

\begin{lem} \label{capillary change}
On $\partial \Sigma$, 	there holds, 
	\begin{align*}
		&\frac{d}{ds}\bigg|_{s=0} \langle \nu(\Sigma(s)),\nu(S)\rangle  \\&\quad =h(S)(\mu(S(\Sigma)),\mu(S(\Sigma)))\,f-\cos(\theta)\,h(\Sigma)(\mu(\Sigma),\mu(\Sigma))\,f-\sin(\theta)\,\langle \nabla^\Sigma f,\mu(\Sigma)\rangle.
	\end{align*}
\end{lem}
\begin{proof}
	Using Lemma \ref{angles}, we have, on $\partial \Sigma$,
	\begin{align*} 
		  	\operatorname{proj}_{T\Sigma} \nu(S)=\sin(\theta)\,\mu(\Sigma)\qquad\qquad \text{ and }\qquad\qquad   \operatorname{proj}_{TS} \nu(\Sigma )=\sin(\theta)\,\mu(S(\Sigma)).
	\end{align*} 
	The assertion follows from this and Lemma \ref{normal change}.
\end{proof}

Recall that 
\begin{align} \label{jacobi equation}
	\frac{d}{ds}\bigg|_{s=0}H(\Sigma(s))=-\Delta^\Sigma f-|h(\Sigma)|^2\,f+\langle \nabla ^\Sigma H(\Sigma),X\rangle.
\end{align}
	\begin{lem} \label{first and second variation} There holds
		\begin{align} \label{area first}  
		\frac{d|\Sigma(s)|}{ds}\bigg|_{s=0}=\int_{\Sigma}H(\Sigma)\,f-\int_{\partial\Sigma} \frac{\cos(\theta)}{\sin(\theta)}\,f
		\end{align} 
		and
		\begin{align} \label{enclosed area first}
		\frac{d|S(\Sigma(s))|}{ds}\bigg|_{s=0}=\int_{\partial \Sigma} \frac{1}{\sin(\theta)}\,f.
		\end{align} 
Moreover, if $H(\Sigma)=0$ and  $\theta$ is constant on $\partial \Sigma$, then
		\begin{equation} \label{area second}
		\begin{aligned} 
&\frac{d^2|\Sigma(s)|}{ds^2}\bigg|_{s=0}+\cos(\theta)\,\frac{d^2|S(\Sigma(s))|}{ds^2}\bigg|_{s=0}\\&\qquad =\int_{\Sigma} |\nabla^\Sigma f|^2-\int_\Sigma|h(\Sigma)|^2\,f^2-\frac{1}{\sin(\theta)}\int_{\partial \Sigma}H(S)\,f^2+\int_{\partial \Sigma}k(\partial \Sigma)\,f^2.
		\end{aligned} 
		\end{equation} 
	\end{lem}
\begin{proof}
\eqref{area first} follows from the first variation of area formula.
\eqref{enclosed area first} follows from the first variation of volume formula, using that $\dot{F}$ is tangent to $S$ on $\partial \Sigma$ and $$\langle \dot{F},\mu(S(\Sigma))\rangle =\frac{1}{\sin(\theta)}\,f.$$

Assume now that $H(\Sigma)=0$ and that $\theta$ is constant on $\partial \Sigma$. Then, using \eqref{jacobi equation}, 
\begin{align*} 
&\frac{d^2|\Sigma(s)|}{ds^2}\bigg|_{s=0}+\cos(\theta)\,\frac{d^2|S(\Sigma(s))|}{ds^2}\bigg|_{s=0}\\&\qquad =-\int_{\Sigma}f\,\Delta^\Sigma  f-\int_{\Sigma}|h(\Sigma)|^2\,f^2\,-\frac{1}{\sin(\theta)}\int_{\partial \Sigma} f\,
\frac{d}{ds}\bigg|_{s=0} \langle \nu(\Sigma_t),\nu(S)\rangle.
\end{align*} 
On $\partial \Sigma$, using that $H(\Sigma)=0$, there holds
\begin{align*} 
h(S)(e(\partial \Sigma),e(\partial \Sigma))&=H(S)-h(S)(\mu(S(\Sigma)),\mu(S(\Sigma)))\text{ and}\\
h(\Sigma)(e(\partial \Sigma),e(\partial \Sigma))&=-h(\Sigma)(\mu(\Sigma),\mu(\Sigma)).
\end{align*} 
In conjunction with \eqref{geodesic} and Lemma \ref{capillary change}, we obtain that 
\begin{align} \label{capillary change 2} 
\frac{d}{ds}\bigg|_{s=0} \langle \nu(\Sigma(s)),\nu(S)\rangle=-\sin(\theta)\,\langle \nabla^\Sigma f, \mu(\Sigma)\rangle -\sin(\theta)\,k(\partial \Sigma)\,f+H(S)\,f.
\end{align} 
The assertion now follows from these identities, using that 
$$
-\int_{\Sigma}f\,\Delta^\Sigma f+\int_{\partial \Sigma}f\,\langle \nabla^\Sigma f,\mu(\Sigma)\rangle =\int_{\Sigma}|\nabla^\Sigma f|^2.
$$
\end{proof}	

	\section{The free energy and minimal capillary surfaces}
	\label{appendix:free energy}
	In this section, we recall background on the free energy and on minimal capillary surfaces.
	
		Recall from Appendix \ref{appendix:af surfaces} the definition of an asymptotically flat support surface $S\subset \mathbb{R}^3$ with unit normal $\nu(S)$, the definition of the region $M(S)$ above $S$, and  the conventions for the second fundamental form $h(S)$ and the mean curvature $H(S)$. 
	
	Let $S\subset \mathbb{R}^3$ be an asymptotically flat support surface diffeomorphic to the plane. Recall from Appendix \ref{appendix:af surfaces} the definition of an admissible surface $\Sigma \subset \bar M(S)$, of its lateral  surface $S(\Sigma)\subset S$, and of its inside $\Omega(\Sigma)\subset M(S)$. Moreover, recall the definitions of the normal $\nu(\Sigma)$, the co-normal $\mu( \Sigma),$ and the co-normal $\mu(S(\Sigma))$, the conventions for the geodesic curvature $k(\partial \Sigma)$, the second fundamental form $h(\Sigma),$ and  the mean curvature $H(\Sigma)$, as well as the definitions of the Gauss curvature $K(\Sigma)$ and the Euler characteristic $\chi(\Sigma)$. Given $f\in C^\infty(\Sigma)$, recall the definitions of $\nabla^\Sigma f$ and $\Delta^\Sigma f$. 
	
	Recall from Appendix \ref{appendix:variation} the definition of an admissible variation $\{\Sigma(s):s\in(-\varepsilon,\varepsilon)\}$ of an admissible surface $\Sigma \subset \bar M(S)$. 
	
	Given $t\in\mathbb{R}$, we define the free energy $J_t(\Sigma)$ with angle $\arccos(-\tanh (t))$ of an admissible surface $\Sigma\subset \bar M(S)$ by
	\begin{align} \label{free energy smooth}
		J_t(\Sigma)=|\Sigma|-\tanh (t)\,|S(\Sigma)|.
	\end{align} 
	We say that an admissible surface $\Sigma\subset \bar M(S)$ is a  minimal capillary surface with capillary angle $\theta \in (0,\pi)$ if $\Sigma$ is a minimal surface and, on  $\partial \Sigma$, $$\langle \nu(\Sigma),\nu(S)\rangle =\cos(\theta).$$ 
	Note that every component of a minimal capillary surface with capillary angle $\theta \in(0,\pi)$ is a minimal capillary surface with capillary angle $\theta\in(0,\pi)$.

	The free energy mass $m_f(\Sigma)$ of a minimal capillary surface $\Sigma\subset \bar M(S)$ with capillary angle $\theta\in(0,\pi)$ is
	\begin{align}\label{free energy mass}
		m_{f}(\Sigma)=\sin(\theta)\,\sqrt{\frac{|\Sigma|}{\pi}}.
	\end{align} 
	
	\begin{lem}
		Let $t\in\mathbb{R}$. An admissible surface $\Sigma\subset \bar M(S)$  is a  minimal capillary surface with capillary angle $\arccos(-\tanh (t))$ if and only if
		$$
		\frac{d J_t(\Sigma(s))}{ds}\bigg|_{s=0}=0
		$$
		 for every admissible variation $\{\Sigma(s):s\in(-\varepsilon,\varepsilon)\}$ of $\Sigma$. 
	\end{lem}
	\begin{proof}
		This follows from Lemma \ref{first and second variation}.
	\end{proof}
	
   For the statement and proof of Lemma \ref{normal variation}, note that 
$
\sin(\arccos(-\tanh (t)))=\cosh (t)^{-1}
$ 
for all $t\in\mathbb{R}$.

\begin{lem} \label{normal variation}
	Let $t\in\mathbb{R}$. Suppose that $\Sigma_t\subset \bar M(S)$ is a  minimal capillary surface with capillary angle $\arccos(-\tanh (t))$ and  that $\Sigma\subset \bar M(S)$ is an admissible surface disjoint from $\Sigma_t$ such that $\Omega(\Sigma)\subset \Omega(\Sigma_t)$. Let $\Sigma^1_t,\ldots, \Sigma^\ell_t$ be the components of $\Sigma_t$ and $\gamma_1,\ldots,\gamma_\ell\in \mathbb{R}$ be such that \begin{align}\label{coefficients sum} \gamma=\sum_{i=1}^\ell \gamma_i\,|\partial \Sigma^i_t|> 0 .\end{align}  	There exists $\delta>0$, a function $\varphi\in C^\infty((|S(\Sigma_t)|-\delta,|S(\Sigma_t)|+\delta)),$ and an admissible variation $$\{\Sigma(s):s\in(|S(\Sigma_t)|-\delta,|S(\Sigma_t)|+\delta)\}$$ of $\Sigma_t$  such that the following properties hold. 
	\begin{itemize}
	%	\item[$\circ$] $\Sigma(|S(\Sigma)|)=\Sigma_t$
		\item[$\circ$] $\varphi'(|S(\Sigma_t)|)=\tanh (t)$
		\item[$\circ$] $\varphi(s)=|\Sigma(s)|$ for every $s\in(|S(\Sigma_t)|-\delta,|S(\Sigma_t)|+\delta)$
		\item[$\circ$] $|S(\Sigma(s))|=s$ for every $s\in(|S(\Sigma_t)|-\delta,|S(\Sigma_t)|+\delta)$
		\item[$\circ$] $\Omega(\Sigma)\subset \Omega(\Sigma(s))$ for every $s\in(|S(\Sigma_t)|-\delta,|S(\Sigma_t)|+\delta)$
	\end{itemize}
	Moreover, we have 
	\begin{align*}
		\varphi''(|S(\Sigma_t)|)&=\gamma^{-2}\,(1-\varphi'(|S(\Sigma_t)|)^2)
		 \sum_{i=1}^\ell \gamma_i^2\left(2\,\pi\,\chi(\Sigma^i_t)-\frac12\,\int_{\Sigma^i_t} |h(\Sigma_t)|^2-\cosh (t)\,\int_{\partial \Sigma^i_t } H(S)\right).
	\end{align*}
	
\end{lem}
\begin{proof}
	Given $\tilde \delta>0$ sufficiently small, let $\{\tilde \Sigma(\tilde s):\tilde s\in(-\tilde \delta,\tilde \delta)\}$ be an admissible variation of $\Sigma_t$ whose normal speed $\dot f^\perp$  at $s=0$ satisfies $\dot f^\perp=\gamma_i$ on $\Sigma^i_t$. Note that  $\Omega(\Sigma)\subset \Omega(\tilde \Sigma(\tilde s))$ for every $s\in(-\tilde \delta,\tilde \delta)$ provided that $\tilde \delta>0$ is sufficiently small.
	
	Let $\tilde \varphi,\,\tilde \eta\in C^\infty((-\tilde \delta,\tilde \delta))$ be given by $\tilde \varphi(\tilde s)=\tilde \Sigma(\tilde s)$ and $\tilde \eta(\tilde s)=|S(\tilde \Sigma(\tilde s))|.$  By Lemma \ref{first and second variation}, 
	\begin{align} \label{first derivative}  
		\tilde \varphi'(0)=\tanh (t)\,\cosh (t)\,\gamma\qquad \text{and}\qquad \tilde \eta'(0)=\cosh (t)\,\gamma.
	\end{align} 
	Moreover, we have
	$$
	\tilde \varphi''(0)-\tanh (t)\,\tilde \eta''(0)=\sum_{i=1}^\ell\,\gamma^2_i \left(-\int_{\Sigma^i_t} |h(\Sigma_t)|^2-\cosh (t)\,\int_{\partial \Sigma^i_t} H(S)+\int_{\partial \Sigma^i_t}k(\partial \Sigma_t)\right).
	$$
	By the Gauss-Bonnet theorem, using that $2\,K(\Sigma_t)=-|h(\Sigma_t)|^2$,
	$$
	-\frac12\,\int_{\Sigma^i_t} |h(\Sigma_t)|^2+\int_{\partial \Sigma^i_t}k(\partial \Sigma_t)= 2\,\pi\,\chi(\Sigma^i_t)
	$$
	for every $1\leq i\leq \ell$.	Shrinking $\tilde \delta>0$, if necessary, and using \eqref{coefficients sum} and \eqref{first derivative}, we see that $\tilde \eta$ is strictly increasing on $(-\tilde \delta,\tilde \delta)$. Let
	\begin{itemize}
		\item[$\circ$] $\delta=\min\{|S(\tilde \Sigma(\tilde \delta))|-|S(\Sigma_t)|,|S(\Sigma_t)|-|S(\tilde\Sigma(-\tilde \delta))|\}$,
		\item[$\circ$] $\varphi\in C^\infty((|S(\Sigma_t)|-\delta,|S(\Sigma_t)|+\delta))$ be given by $\varphi(s)=\tilde \varphi(\tilde \eta^{-1}(s))$, and
		\item[$\circ$]  $\Sigma(s)=\tilde \Sigma(\tilde \eta^{-1}(s))$ for $s\in  (|S(\Sigma_t)|-\delta,|S(\Sigma_t)|+\delta)$.
	\end{itemize}   From this,
	we see that  $\delta$, $\varphi$, and $\{\Sigma(s):s\in (|S(\Sigma_t)|-\delta,|S(\Sigma_t)|+\delta)\}$ have all of the asserted properties. 
\end{proof}

	A minimal capillary surface $\Sigma_t\subset \bar M(S)$ with capillary angle $\arccos(-\tanh (t))$ is called stable if
	\begin{align} \label{stability ineq} 
	\int_{\Sigma_t} |\nabla^{\Sigma_t} f|^2-\int_{\Sigma_t}|h(\Sigma_t)|^2\,f^2-\cosh (t)\int_{\partial \Sigma_t}H(S)\,f^2+\int_{\partial \Sigma_t}k(\partial \Sigma_t)\,f^2\geq 0
	\end{align} 	
	for every $f\in C^\infty(\Sigma_t)$.
	Note that $\Sigma_t$ is stable if and only if every component of $\Sigma_t$ is stable.
	\begin{lem} \label{stability}
		Let  $t\in\mathbb{R}$. A minimal capillary surface $\Sigma_t\subset \bar M(S)$  with capillary angle $\arccos(-\tanh (t))$  is stable if and only if
		$$
		\frac{d^2J_t(\Sigma(s))}{ds^2}\bigg|_{s=0}\geq 0
		$$
		 for every admissible variation $\{\Sigma(s):s\in(-\varepsilon,\varepsilon)\}$ of $\Sigma_t$.		
	\end{lem}
\begin{proof}
	This follows from Lemma \ref{first and second variation}.
\end{proof}
\begin{lem} \label{first eigenvalue}
	Let $\Sigma_t\subset \bar M(S)$ be a connected stable minimal capillary surface with capillary angle $\arccos(-\tanh (t))$.
	Let $\kappa_t\in \mathbb{R}$ be the infimum of the functional
	\begin{align} \label{functional} 
		f\mapsto 	\int_{\Sigma_t} |\nabla^{\Sigma_t} f|^2-\int_{\Sigma_t}|h(\Sigma_t)|^2\,f^2-\cosh (t)\,\int_{\partial \Sigma_t}H(S)\,f^2+\int_{\partial \Sigma_t}k(\partial \Sigma_t)\,f^2
	\end{align} 
	among all $f\in W^{1,2}(\Sigma_t)$ with  
	$$
	\int_{\Sigma_t} f^2=1.	$$
There holds $\kappa_t\geq 0$ and there is a positive function  $f_t\in C^\infty(\Sigma_t)$    such that
	\begin{align} \label{first eigenfunction} 
	\begin{dcases}
		-\Delta^{\Sigma_t} f_t-|h(\Sigma_t)|^2\,f_t=\kappa_t\,f_t\qquad&\text{ in }\Sigma_t\text{ and}
		\\-\langle \nabla ^{\Sigma_t} f_t,\mu(\Sigma_t)\rangle -k(\partial \Sigma_t)\,f_t+\cosh (t)\,H(S)f_t=0 &\text{ on }\partial \Sigma_t
	\end{dcases}
	\end{align} 
	Moreover, if $\tilde f\in C^\infty(\Sigma_t)$ satisfies 
	\begin{align*} 
		\begin{dcases}
		-\Delta^{\Sigma_t} \tilde f-|h(\Sigma_t)|^2\,\tilde f=\kappa_t\,\tilde f\qquad&\text{ in }\Sigma_t\text{ and}
		\\-\langle \nabla ^{\Sigma_t} \tilde f,\mu(\Sigma_t)\rangle -k(\partial \Sigma_t)\,\tilde f+\cosh (t)\,H(S)\tilde f=0 &\text{ on }\partial \Sigma_t,
	\end{dcases}
	\end{align*} 
	 then $\tilde f=\beta\,f_t$ for some $\beta\in\mathbb{R}$. 
\end{lem}
\begin{proof}
	This follows from minimizing \eqref{functional} 
	among all $f\in W^{1,2}(\Sigma_t)$ with  
	$$
	\int_{\Sigma_t} f^2=1,	$$
	using the Rellich–Kondrachov theorem, the trace theorem for Sobolev functions, elliptic regularity, and \eqref{stability ineq}; see, e.g., \cite[Theorem 2 in \S6.5]{Evans}.
\end{proof}
	\begin{lem} \label{minimization in homology class3}
	Let $t\in \mathbb R$. 	Suppose that   $\Sigma_t\subset \bar M(S)$ is a stable minimal capillary surface with capillary angle $\arccos(-\tanh (t))$ such that $H(S)\geq 0$ on $\partial \Sigma_t$. Then  $\Sigma_t$ is  diffeomorphic to a union of disks.
	\end{lem}
	\begin{proof}
		The argument is essentially the same  as that in the proof of \cite[Proposition 2]{Ros}. 
		Given a component  $\Sigma^1_t\subset \Sigma_t$, we use  $f=\mathbbm{1}_{\Sigma^1_t}$ in \eqref{stability ineq}. Using that $2\,K(\Sigma_t)=-|h(\Sigma_t)|^2$ and the Gauss-Bonnet theorem
		$$
		\int_{\Sigma^1_t} K(\Sigma_t)+\int_{\partial \Sigma^1_t} k(\partial \Sigma_t)=2\,\pi\,\chi(\Sigma^1_t),
		$$
		  we obtain
		$$
		\frac12\,\int_{\Sigma^1_t} |h(\Sigma_t)|^2+\cosh (t)\,\int_{\partial \Sigma^1_t}H(S)\leq 2\,\pi\,\chi(\Sigma^1_t).
		$$
		In particular, $\chi(\Sigma^1_t)\geq 0$. It follows that $\Sigma^1_t$ is  diffeomorphic to either a disk or an annulus. If $\Sigma^1_t$ is diffeomorphic to an annulus, i.e., $\chi(\Sigma^1_t)=0$, it follows that $H(S)=0$ on $\partial \Sigma^1_t$ and  $h(\Sigma_t)=0$  on $\Sigma^1_t$. Consequently, equality holds in \eqref{stability ineq} for $f=\mathbbm{1}_{\Sigma^1_t}$. Applying Lemma \ref{first eigenvalue}, we see that  
		$$
		\begin{dcases}
			-\Delta^{\Sigma_t} f=0\qquad&\text{in }\Sigma^1_t\text{ and}\\
			-\langle \nabla^{\Sigma_t} f,\mu(\Sigma_t)\rangle -k(\partial \Sigma_t)\,f=0&\text{on }\partial \Sigma^1_t.
		\end{dcases}
		$$
		Thus, $k(\partial \Sigma_t)=0$ on $\partial \Sigma^1_t$. Since 
		$$
		\int_{\partial \Sigma^1_t}|k(\partial \Sigma_t)|\geq 4\,\pi,
		$$  
		we obtain a contradiction.
	\end{proof}
	\begin{lem} \label{smoothing the corner}
		Let $t\in\mathbb{R}$. Suppose that $\Sigma_t\subset \bar M(S)$ is a stable minimal capillary surface with capillary angle $\arccos(-\tanh (t))$  such that 
		$H(S)\geq 0$  on $S\setminus S(\Sigma_t)$. There exists an asymptotically flat support surface $S_t\subset\mathbb{R}^3$ diffeomorphic to the plane with $H(S_t)\geq 0$ such that $(S\setminus S_t)\cup (S_t \setminus S)$ is compact, $M( S_t)\subset M(S)\setminus \bar \Omega(\Sigma_t)$, and $S_t$ is not a flat plane.
	\end{lem}
	\begin{proof}
		According to Lemma \ref{minimization in homology class3}, $\Sigma_t$ is diffeomorphic to a union of disks. Consequently, $(S\setminus S(\Sigma_t))\cup \Sigma_t$ is homeomorphic to the plane. The assertion now follows from smoothing  $(S\setminus S(\Sigma_t))\cup \Sigma_t$ near the edge $\partial \Sigma_t$, using that $H(S)\geq 0$ on $S\setminus S(\Sigma_t)$ and that $H(\Sigma_t)=0$. 
	\end{proof}

	The proof of the following lemma is inspired by that of \cite[Theorem 3]{MeeksYau}; see also \cite[Theorem 3.1]{Galloway}.

	\begin{lem} \label{local foliation}
		Let $t\in\mathbb{R}$ and suppose that $\Sigma_t \subset \bar M(S)$ is a connected stable minimal capillary surface with capillary angle $\arccos(-\tanh (t))$. There exist $\delta>0$, $\omega\in C^\infty((-\delta,\delta))$, and  a smooth foliation $\{\Sigma(s):s\in(-\delta,\delta)\}$ by minimal capillary surfaces $\Sigma(s)\subset \bar M(S)$ with capillary angle $\arccos(-\tanh(\omega(s)))$ such that the following properties hold. There holds $\Sigma(0)=\Sigma_t$, $\omega'(0)\geq 0$,  and the normal speed  $\dot f^\perp \in C^\infty(\Sigma_t)$ of the foliation at $s=0$ is positive and satisfies   
		\begin{align} \label{speed} 
			\begin{dcases}
				-\Delta^{\Sigma_t} \dot f^\perp -|h(\Sigma_t)|^2\,\dot f^\perp =0\qquad&\text{ in }\Sigma_t\text{ and}
				\\-\langle \nabla ^{\Sigma_t} \dot f^\perp ,\mu(\Sigma_t)\rangle -k(\partial \Sigma_t)\,\dot f^\perp +\cosh (t)\,H(S)\,\dot f^\perp =-\cosh (t)^{-1}\,\omega'(0) &\text{ on }\partial \Sigma_t.
			\end{dcases}
		\end{align} 
	\end{lem}
	\begin{proof}
		Since $\langle \nu(\Sigma_t),\nu(S)\rangle =\tanh (t)$, there are $\varepsilon>0$ and  $F\in C^\infty(\Sigma_t\times(-\varepsilon,\varepsilon),\mathbb{R}^3)$ such that
		\begin{equation*}
			\begin{aligned} 
				&\circ\qquad\text{$F$ is a diffeomorphism with its image,}\\
				&\circ\qquad\text{$\{x\in \bar M(S):\operatorname{dist}(x,\Sigma_t)<\cosh( t)^{-1}\,\varepsilon\}\subset \{F(y,s):y\in \Sigma_t$ and $s\in(-\varepsilon,\varepsilon)\}$, and}
				\\&\circ\qquad  \langle \dot F,\nu(\Sigma_t ) \rangle\geq \frac{1}{\cosh (t)} \text{ where } \dot F=\frac{dF(\,\cdot\,,s)}{ds}\bigg|_{s=0}.
			\end{aligned} 
		\end{equation*}
		Given $\alpha\in(0,1)$ and $\delta>0$, let $$\mathcal{U}_\delta=\{f\in C^{2,\alpha}(\Sigma_t):|f|_{C^{2,\alpha}(\Sigma_t)}<\delta\}.$$ Provided that $\delta>0$ is sufficiently small,  $\Sigma(f)=\{F(y,f(y)):y\in \Sigma_t\}$ is an admissible surface for every $f\in \mathcal{U}_\delta$. We tacitly identify functions and maps defined on $\Sigma(f)$ with functions and maps defined on $\Sigma_t$ by precomposition with the map $\Sigma_t\to \Sigma(f)$ given by $y\mapsto F(y,f(y))$. Given $f\in C^{2,\alpha}(\Sigma_t)$, let $f^\perp=\langle \dot F,\nu(\Sigma_t)\rangle\, f$. Note that $f=0$ if and only if $f^\perp=0$. 
		
		According to Lemma \ref{first eigenvalue}, there is a nonnegative number $\kappa_t$ and a positive function $f_t\in C^\infty(\Sigma_t)$ that satisfy \eqref{first eigenfunction}.
		
		We first assume that $\kappa_t>0$. Let  $\mathcal{F}_\delta:\mathcal{U}_\delta\to C^{0,\alpha}(\Sigma_t)\times C^{1,\alpha}(\partial \Sigma_t)$ be given by  
		$$
		\mathcal F_\delta(f)=(H(\Sigma(f)),\langle \nu(\Sigma(f)),\nu(S)\rangle).
		$$	
		In view of \eqref{jacobi equation} and \eqref{capillary change 2}, the differential $\mathcal{D}\mathcal{F_\delta}(0) : C^{2,\alpha}(\Sigma_t)\to C^{0,\alpha}(\Sigma_t)\times C^{1,\alpha}(\partial \Sigma_t)$ is given by
		$$
		\mathcal{D}_f\mathcal{F_\delta}(0)=\left (	-\Delta^{\Sigma_t} f^\perp-|h(\Sigma_t)|^2\,f^\perp,-\frac{1}{\cosh (t)}\,\langle \nabla^{\Sigma_t} f^\perp, \mu(\Sigma_t)\rangle -\frac{1}{\cosh (t)}\,k(\partial \Sigma_t)\,f^\perp+H(S)\,f^\perp \right ).
		$$
		By the Fredholm alternative and Lemma \ref{first eigenvalue}, using that $\kappa_t>0$, we see that $\mathcal{D}\mathcal{F}_\delta(0):C^{2,\alpha}(\Sigma_t)\to C^{0,\alpha}(\Sigma_t)\times C^{1,\alpha}(\partial \Sigma_t)$ is an isomorphism. Consequently, there is a smooth family $\{f(s):s\in(-\delta,\delta)\}$ of functions $f(s)\in C^{2,\alpha}(\Sigma_t)$ such that $\Sigma(f(s))$ is a capillary minimal surface with capillary angle $\arccos(-\tanh(t+s))$ provided that $\delta>0$ is sufficiently small. Using \eqref{jacobi equation} and \eqref{capillary change 2}, we see that the normal speed $\dot f^\perp$ of this family at $s=0$ satisfies \eqref{speed} with $\omega(s)=t+s$. Moreover, the function $q= f_t^{-1}\,\dot f^\perp$ satisfies 
		\begin{align*} 
			\begin{dcases}
				-\Delta^{\Sigma_t} q -f_t^{-1}\,\langle \nabla^{\Sigma_t}q,\nabla^{\Sigma_t} f_t\rangle  =-\kappa_t\,q -f_t^{-2}\,|\nabla^{\Sigma_t}f_t|^2\,q &\text{ in }\Sigma_t\text{ and}
				\\-\langle \nabla ^{\Sigma_t} q ,\mu(\Sigma_t)\rangle  =-\cosh (t)^{-1}\,f_t &\text{ on }\partial \Sigma_t.
			\end{dcases}
		\end{align*} 
		By the strong maximum principle, $q>0$ and, consequently, $\dot f^\perp >0$. It follows that  $\{\Sigma(f(s)):s\in(-\varepsilon,\varepsilon)\}$ is a smooth foliation. This completes the proof in the case where $\kappa_t>0$.
		
		Assume now that $\kappa_t=0$. Let $\tilde{ \mathcal{F}}_\delta:\mathcal{U}_\delta\times \mathbb{R}\to C^{0,\alpha}(\Sigma_t)\times C^{1,\alpha}(\partial \Sigma_t)\times\mathbb{R}$ be given by  
		$$
		\tilde {\mathcal F}_\delta(f,\omega)=\left(H(\Sigma(f)),\langle \nu(\Sigma(f)),\nu(S)\rangle+\tanh(\omega),\int_{\Sigma_t}f_t\, f^\perp+\int_{\partial \Sigma_t} f_t\, f^\perp \right).
		$$
		By the Fredholm alternative, given $a \in C^{0,\alpha}(\Sigma_t)$ and $b \in C^{1,\alpha}(\partial \Sigma_t)$, there is a unique $f\in C^{2,\alpha}(\Sigma_t)$ with 
		\begin{align*} 
			\begin{dcases}
				-\Delta^{\Sigma_t} f-|h(\Sigma_t)|^2\,f=a\qquad&\text{ in }\Sigma_t\text{ and}
				\\-\langle \nabla ^{\Sigma_t} f,\mu(\Sigma_t)\rangle -k(\partial \Sigma_t)\,f+\cosh (t)\,H(S)f=b &\text{ on }\partial \Sigma_t
			\end{dcases}
		\end{align*} 
		and 
		$$
		\int_{\Sigma_t} f_t\,f+\int_{\partial \Sigma_t} f_t\,f=0
		$$
		if and only if 
		\begin{align} \label{necessary condition}
			\int_{\Sigma_t} f_t\,a+\int_{\partial \Sigma_t} f_t\,b=0.
		\end{align}
		It follows that the differential $\mathcal{D}\tilde {\mathcal{F_\delta}}(0) : C^{2,\alpha}(\Sigma_t)\times \mathbb{R}\to C^{0,\alpha}(\Sigma_t)\times C^{1,\alpha}(\partial \Sigma_t)\times \mathbb{R}$  is an isomorphism. Consequently, there is a smooth family $\{f(s):s\in(-\delta,\delta)\}$ of functions $f(s)\in C^{2,\alpha}(\Sigma_t)$ and $\omega\in C^\infty((-\delta,\delta))$ such that $\Sigma(f(s))$ is a minimal capillary surface with capillary angle $\arccos(-\tanh(\omega(s)))$ and 
		$$
		\int_{\Sigma_t} f_t\,f(s)^\perp +\int_{\partial \Sigma_t} f_t\,f(s)^\perp=s
		$$
for all $s\in(-\delta,\delta)$. 		As before, we see that the normal speed $\dot f^\perp$ of this family at $s=0$ satisfies \eqref{speed} and that
		\begin{align} \label{equation}  
			\int_{\Sigma_t} f_t\,\dot f^\perp +\int_{\partial \Sigma_t} f_t\,\dot f^\perp=1.
		\end{align} 
		In view of \eqref{necessary condition}, we have $\omega'(0)=0$. Using Lemma \ref{first eigenvalue} and \eqref{equation}, we conclude that $\dot f^\perp=\beta\,f_t$ for some $\beta>0$. This completes the proof in the case where $\kappa_t=0$.	
	\end{proof}

\section{Sets of finite perimeter}
\label{BV}
Let $\Omega\subset \mathbb{R}^3$ be an open set  with Lipschitz boundary.

Following \cite[\S5]{EvansGariepy}, we say that a measurable set $E\subset \Omega$ has finite perimeter in $\Omega$ if $\mathbbm{1}_E\in L^1(\Omega)$ and 
$$
||\partial E||(\Omega)=\sup\left\{\int_E \operatorname{div}\varphi:\varphi\in C^\infty_c(\Omega,\mathbb{R}^n)\text{ with }|\varphi|\leq 1\right\}<\infty.
$$
 Recall from \cite[\S5.1]{EvansGariepy} that given a set  $E\subset \Omega$ of finite perimeter in $\Omega$ there is a  Radon measure $||\partial E||$ on $\Omega$ called the perimeter measure of $E$ in $\Omega$ such that, for every open set $\Omega'\subset \Omega$ that is compactly contained in $\Omega,$ we have 
$$
||\partial E||(\Omega')=\sup\left\{\int_E \operatorname{div}\varphi: \varphi\in C^\infty_c(\Omega',\mathbb{R}^n)\text{ and }|\varphi|\leq 1\right\}.
$$
 Recall from \cite[\S5.7]{EvansGariepy} that if $E\subset \Omega$ is a set of finite perimeter in $\Omega$, there exists a  2-rectifiable set $\partial^* E\subset  \Omega$ called the reduced boundary of $E$ in $\Omega$ such that 
$$
||\partial E||(\Omega\setminus \partial ^*E)=0\qquad\text{and}\qquad ||\partial E||(\Omega)=|\partial^* E|.
$$
If $F \subset \mathbb{R}^3$ is a smooth bounded set such that $\partial F$ and $\partial \Omega$ intersect transversely, then the reduced boundary of $\Omega\cap F$ in $\Omega$ is equal to $  \Omega\cap \partial F$.
\begin{lem}[{\cite[Theorem 4 in \S5.2]{EvansGariepy}}] \label{bv:lower}
	Let $\{E_k\}_{k=1}^\infty$ be a sequence of sets $E_k\subset \Omega$ of finite perimeter in $\Omega$  and suppose that $\mathbbm{1}_{E_k}\to \mathbbm{1}_E$ in $L_{loc}^1(\Omega)$ for some set $E\subset \Omega$ of finite perimeter in $\Omega$. Then 
	$$
	||\partial E||(\Omega)\leq \liminf_{k\to\infty} ||\partial E_k||(\Omega).
	$$
\end{lem}

\begin{lem}[{\cite[Theorem 4 in \S5.2]{EvansGariepy}}] \label{bv:compact} Let $K\subset \mathbb{R}^3$ be  a compact set.
	Suppose that $\{E_k\}_{k=1}^\infty$ is a sequence of sets $E_k\subset K\cap \Omega$ of finite perimeter in $\Omega$  with
$$
\sup_{k\geq 1} ||\partial E_k||(\Omega)<\infty.
$$
Then there is a set $E\subset K\cap \Omega$ of finite perimeter in $\Omega$  such that, passing to a subsequence, $\mathbbm{1}_{E_k}\to \mathbbm{1}_E$ in $L^1(\Omega)$. 
\end{lem}

Each  set $E\subset \Omega$ of finite perimeter in $\Omega$ has a trace $\mathbbm{1}_E|_{\partial \Omega}\in L_{loc}^1(\partial \Omega)$ such that, for almost every $x\in \partial \Omega$, 
 \begin{align*} %\label{trace definition}
 	\mathbbm{1}_E|_{\partial \Omega}(x)=\lim_{r\searrow 0}\frac{|\{x'\in E:|x'-x|<r\}|}{|\{x'\in \Omega:|x'-x|<r\}|}; \end{align*}
 see \cite[Theorem 1 in \S5.3]{EvansGariepy} and \cite[Theorem 2 in \S5.3]{EvansGariepy}. Consequently,
if $F \subset \mathbb{R}^3$ is a smooth bounded set such that $\partial F$ and $\partial \Omega$ intersect transversely, then $\mathbbm{1}_{\Omega\cap F}|_{\partial \Omega}=\mathbbm{1}_{\partial \Omega\cap F}$ almost everywhere on $\partial \Omega$.  
  
\begin{lem}[{\cite[Theorem 1 in \S5.3]{EvansGariepy}}]  \label{bv:trace1} Let $E\subset \Omega$ be a bounded set of finite perimeter in $\Omega$. There holds $\mathbbm{1}_E|_{\partial \Omega}\in L^1(\partial \Omega)$. Moreover,  there holds   
	$$
	\int_{\partial \Omega} |\mathbbm{1}_E|_{\partial \Omega}|\,|\varphi|\leq \int_{\Omega} |\varphi|\,\mathrm{d}\,||\partial E||+\int_{E} |D\varphi|
	$$
	for every $\varphi \in C^\infty(\mathbb{R}^n,\mathbb{R}^n)$.
\end{lem} 
\begin{lem}[{\cite[Theorem 1 in \S5.3]{EvansGariepy}}] \label{bv:trace2} Let $K\subset \mathbb{R}^3$ be a compact set. Suppose that $\{E_k\}_{k=1}^\infty$ is a sequence of sets $E_k\subset K\cap \Omega$ of finite perimeter in $\Omega$  such that 
	$$
\lim_{k\to\infty} \int_{\Omega}|\mathbbm{1}_{E_k}-\mathbbm{1}_E|=0\qquad\text{and}\qquad 	\lim_{k\to\infty}||\partial E_k||(\Omega)=||\partial E||(\Omega) 
	$$
	for some set $E\subset K\cap  \Omega$ of finite perimeter in $\Omega$. Then $\mathbbm{1}_{E_k} |_{\partial \Omega}\to \mathbbm{1}_{E}|_{\partial \Omega}$ in $L^1(\partial \Omega)$. 
\end{lem}

\begin{lem}[{\cite[Lemma 12.22]{Maggi}}] \label{intersection and union}
Let $E,\,F\subset \Omega$ be sets of finite perimeter in $\Omega$. Then $E\cup F$ and $E\cap F$ are sets of finite perimeter in $\Omega$. %and there holds 
%$$
%|\partial^* (E\cap F)|+|\partial^*(E\cup F)|\leq |\partial^* E|+|\partial^*F|.
%$$ 
\end{lem}
\begin{lem}[{\cite[Theorem 13.8 and Remark 13.9]{Maggi}}]\label{bv:perimeterapprox}
		Let $E\subset \Omega$ be a bounded set of finite perimeter in $\Omega$. There exists a compact set $K\subset \mathbb{R}^3$ and a  sequence $\{E_k\}_{k=1}^\infty$ of smooth open sets $E_k\subset K$ such that $\mathbbm{1}_{E_k\cap \Omega}\to  \mathbbm{1}_E$ in $L^1(\Omega)$ and 
	$$
	||\partial(\Omega\cap E_k)||(\Omega)\to||\partial E||(\Omega).
	$$ 
\end{lem}

\end{appendices} 

\end{document}